\documentclass[12pt]{amsart}

\usepackage[left=3cm,top=2.5cm,bottom=2.5cm,right=3cm]{geometry}
\usepackage{times}
\usepackage{amssymb}
\usepackage{graphicx,xspace}
\usepackage{epsfig}
\usepackage{enumitem}
\usepackage[usenames,dvipsnames]{xcolor}
\usepackage{tikz}
\usepackage{gensymb}
\usepackage[T1]{fontenc}
\usepackage[utf8]{inputenc}
\usepackage{bm}
\usepackage{genyoungtabtikz}
\usepackage[config, labelfont={normalsize}]{caption,subfig}
\usepackage{mathrsfs}
\definecolor{green}{RGB}{0,127,0}
\definecolor{red}{RGB}{191,0,0}
\usepackage[colorlinks,cite color=red,link color=green,pagebackref=true]{hyperref}
\usepackage{todonotes}

\usetikzlibrary{matrix,arrows,calc}
\usetikzlibrary{decorations.markings}
\usetikzlibrary{patterns}
\usetikzlibrary{snakes}
\usetikzlibrary{shapes}

\usepackage[capitalize]{cleveref}

\theoremstyle{plain}

\newtheorem{lemma}{Lemma}[section]
\newtheorem{theorem}[lemma]{Theorem}
\newtheorem{corollary}[lemma]{Corollary}
\newtheorem{proposition}[lemma]{Proposition}

\newtheorem{conjecture}[lemma]{Conjecture}

\theoremstyle{remark}
\newtheorem*{remark}{Remark}

\newtheorem{definition}[lemma]{Definition}

\newcommand{\G}{\mathcal{G}}
\newcommand{\A}{\mathcal{A}}

\newcommand{\N}{\mathbb{N}}

\newcommand{\Sym}[1]{\mathfrak{S}_{#1}}

\newcommand{\I}{\mathcal{I}}

\newcommand{\Z}{\mathbb{Z}}

\newcommand{\QQ}{\mathbb{Q}}

\newcommand{\XX}{\bm{X}}
\newcommand{\xx}{\bm{x}}
\newcommand{\yy}{\bm{y}}

\newcommand{\PPP}{\mathcal{P}}

\DeclareMathOperator{\GL}{GL}
\def\la{\lambda}
\def\ka{\kappa}

\def\si{\sigma}

\DeclareMathOperator{\Tu}{Tutte}
\DeclareMathOperator{\Hilb}{Hilb}
\DeclareMathOperator{\Fr}{Fr}

\newcommand{\zz}{\bm{z}}

\DeclareMathOperator{\ia}{ia}
\DeclareMathOperator{\ea}{ea}
\DeclareMathOperator{\inv}{inv}
\DeclareMathOperator{\Inv}{Inv}
\DeclareMathOperator{\InvP}{InvT}

\DeclareMathOperator{\maj}{maj}

\DeclareMathOperator{\Des}{Des}

\DeclareMathOperator{\wt}{wt}

\DeclareMathOperator{\iDes}{iDes}
\DeclareMathOperator{\outdeg}{outdeg}
\DeclareMathOperator{\level}{level}

\def\XX{\bm{X}}
\def\YY{\bm{Y}}

\author[M.~Dołęga]{Maciej Dołęga}
\address{
Wydział Matematyki i Informatyki, 
Uniwersytet im.~Adama Mickiewicza, 
Collegium Mathematicum,
Umultowska 87, 
61-614 Poznań, 
Poland, \newline \indent Instytut Matematyczny,
Uniwersytet Wrocławski,  \mbox{pl.\ Grunwaldzki~2/4,} 50-384
Wrocław, Poland}
\email{maciej.dolega@amu.edu.pl}

 \thanks{
MD is supported from {\it Narodowe Centrum Nauki}, grant UMO-2015/16/S/ST1/00420.}



\title[Macdonald
  cumulants, $G$-inversion polynomials and $G$-parking functions]{Macdonald
  cumulants, $G$-inversion polynomials and $G$-parking functions}

\begin{document}

\maketitle

\begin{abstract}
We prove a combinatorial formula for Macdonald cumulants which
generalizes the celebrated formula of Haglund, Haiman and Loehr for Macdonald
polynomials. We provide several applications of our formula. Firstly,
it allows us to give a new, constructive
proof of a strong factorization property
of Macdonald polynomials proven recently by the author of this
paper. Moreover, we prove that Macdonald cumulants
are $q,t$--positive in the monomial and in the fundamental
quasisymmetric bases. Furthermore, we use our formula to prove the recent higher-order
Macdonald positivity conjecture for the coefficients
of the Schur polynomials indexed by hooks. Our combinatorial formula
relates Macdonald cumulants to the generating function of $G$-parking
functions, or equivalently to a certain specialization of the Tutte
polynomials.
\end{abstract}

\section{Introduction}
\label{sec:introduction}

\subsection{Schur--positivity and Macdonald polynomials}
\label{subsec:SchurPositivity}

An ubiquitous problem in the theory of symmetric functions is to
express a given symmetric function $f$ in a given basis $\{s_\lambda\}_\lambda$ of particular
interest. In other words, one can ask:
\emph{what can we say about the coefficients $a_\lambda$ in the following expansion:}
\[ f = \sum_{\lambda}a_\lambda\ s_\lambda?\] 
Most often, such a basis is the basis of Schur symmetric
functions, which turns out to be the most natural in many different contexts such as the representation
theory of the symmetric groups, algebraic geometry, or random discrete
models, among others. Indeed, Schur symmetric
functions are characters of irreducible representations of $\GL_n$,
their asymptotic behaviour describe many random processes and they
represent Schubert classes in Grassmannian varieties, see \cite{Fulton1997,BufetovGorin2016}. In many cases it was observed that a beautiful but
notoriously difficult to prove phenomenon occurs: all the coefficients
$a_\lambda$ are \emph{nonnegative integers}, or in a more general
setting they are \emph{polynomials with nonnegative integer
  coefficients}. For instance, the homogeneous symmetric function $h_\mu$
which has positive integer coefficients in the Schur basis expansion has a natural
$q$--deformation given by the modified Hall-Littlewood
symmetric function $Q'_\mu(q)$. The coefficients of $Q'_\mu(q)$ expanded in
the Schur basis are polynomials in $q$ with nonnegative integer
coefficients. This phenomenon, called \emph{Schur--positivity}, has a deep geometric reason and a beautiful combinatorial
interpretation in the
case of Hall--Littlewood symmetric
functions \cite{LascouxSchutzenberger1978,Lusztig1981}. Moreover, in
the case of other symmetric functions Schur--positivity builds deep connections between many different areas of
mathematics and physics such as the representation theory of groups,
Hecke algebras, algebraic geometry, or the theory of quantum groups,
see for instance \cite{LascouxLeclercThibon1996,KnutsonTao1999,Haiman2001,LamPostnikovPylyavskyy2007,ShareshianWachs2012}. Therefore, deciding whether
a given symmetric function is Schur--positive is one of the major questions in
the contemporary algebraic combinatorics of symmetric functions.

\vspace{10pt}

One of the most prominent examples of Schur--positive
symmetric functions which contain the aforementioned modified Hall--Littlewood symmetric functions as a special case 
is the Macdonald symmetric function $\tilde{H}_\mu(\xx;q,t)$,
introduced by Macdonald in 1988 \cite{Macdonald1988,
  Macdonald1995} (here, we use ``the transformed form'' of Macdonald
polynomials sometimes called ``the modified form'' introduced by
Garsia and Haiman - see \cite{GarsiaHaiman1993} for its initial definition and a relation with other
forms of Macdonald polynomials). Strictly from the definition, this is a symmetric function
in variables $\xx := x_1,x_2,\dots$ with the coefficients being rational functions in $q,t$.
However, Macdonald conjectured \cite{Macdonald1988} that expanding it in the Schur basis:
\[ \tilde{H}_\mu(\bm{x};q,t) =
  \sum_{\lambda}\tilde{K}_{\lambda,\mu}(q,t)\ s_\la(\xx)\]
the coefficients $\tilde{K}_{\lambda,\mu}(q,t)$ (called \emph{transformed
  $q,t$--Kostka coefficients}) are in fact polynomials in $q,t$ with
nonnegative coefficients. In the following Garsia and Haiman gave a
conjectural representation--theoretic interpretation of the transformed
$q,t$-Kostka coefficients \cite{GarsiaHaiman1993}, and it took almost
ten years more to Haiman to prove it \cite{Haiman2001}. He achieved
this goal
by connecting a representation theoretic interpretation of the transformed
$q,t$-Kostka coefficients with the problem from algebraic geometry of the Hilbert
scheme of $n$ points in the plane. This result is considered as a
great breakthrough in the symmetric functions theory and initiated
very active research in the remarkable algebraic combinatorics of the
Macdonald polynomials, see the expository textbook of Haglund \cite{Haglund2008}.


Macdonald positivity ex-conjecture has seen many generalizations in
different directions up to
these days. One example of such a generalization called \emph{higher--order Macdonald
  positivity conjecture} was presented
in our recent work \cite{Dolega2017} and will be the main subject of
this paper.

\subsection{Cumulants and higher--order Macdonald
  positivity conjecture}
\label{subsec:HigherOrder}

\subsubsection{Cumulants}
A classical problem in the
symmetric functions theory, which is related to the
positivity problem from \cref{subsec:SchurPositivity}, is to understand the so-called \emph{structure
  constants} $a_{\mu,\nu}^\la $ of a given linear basis $\{s_\mu\}_\mu$:
\[ s_\mu\cdot s_\nu = \sum_{\la} a_{\mu,\nu}^\la\ s_\la.\]
Let us look on the structure constants for Macdonald
polynomials. For partitions $\lambda = (\lambda_1,\lambda_2,\dots)$ and $\mu =
(\mu_1,\mu_2,\dots)$ we define a new partition
$\la\oplus\mu := (\la_1+\mu_1,\la_2+\mu_2,\dots)$ by adding
coordinates of partitions $\la$ and $\mu$.
Since Macdonald polynomials $\{\tilde{H}_\mu\}_\mu$ form a linear
basis of the algebra $\Lambda$ of symmetric functions over $\QQ(q,t)$,
we can define a multiplication $\oplus$ on $\Lambda$ by setting $\tilde{H}_\mu
\oplus \tilde{H}_\nu := \tilde{H}_{\mu\oplus\nu}$ and extending it by
linearity. Macdonald showed \cite{Macdonald1995} that
algebras $(\Lambda,\oplus)$ and $(\Lambda,\cdot)$ coincide in the
specialization $q=1$  (this also follows from Haglund, Haiman and Loehr's formula \eqref{eq:Haglund}).
Thus, the much simpler algebraic structure $(\Lambda,\oplus)$ can be
interpreted as an approximation of the algebra $(\Lambda,\cdot)$ of
interest, as $q \to 1$, therefore it is desirable to understand the
difference between these two product structures. A natural way
of measuring the discrepancy between two algebraic structures is
provided by \emph{conditional cumulants}.

\vspace{10pt}

Let $\A$ be a commutative ring with two different multiplicative
structures $\cdot$
and $\oplus$ which define two (different) algebra structures
on $\A$. For any $X_1,\dots,X_r \in \A$ one can define a conditional
cumulant $\ka(X_1,\dots,X_r) \in \A$ as the coefficient of $t_1\cdots t_r$ in the following formal power
series in $t_1,\dots,t_k$:
\begin{equation} 
\label{eq:cumu1}
\ka(X_1,\dots,X_r) := [t_1\cdots t_r] \log_{\cdot} \left(
  \exp_{\oplus}(t_1X_1+\cdots t_rX_r)\right),
\end{equation}
where $\log_{\cdot}$ and $\exp_{\oplus}$ are defined in a standard way
with respect to multiplication given by $\cdot$ and $\oplus$
respectively. Thus
\[ \log_{\cdot}(1+A) = \sum_{n \geq 1}\frac{(-1)^{n-1}A^{\cdot
      n}}{n},\]
and 
\[ \exp_{\oplus}(A) = \sum_{n \geq 0}\frac{A^{\oplus n}}{n!}.\] 
Definition \eqref{eq:cumu1} can be transformed into an equivalent
but more
combinatorial definition:
\begin{equation} 
\label{eq:DefCumulants}
\ka(X_1,\dots,X_r) = \sum_{\substack{\pi \in \PPP([r]) } } 
            (-1)^{\#\pi-1}(\#\pi-1)! \prod_{B \in \pi} \bigoplus_{b \in B}X_b.
\end{equation}
Here we sum over \emph{set-partitions} of $[r] := \{1,2,\dots,r\}$, that is all possible
sets $\pi$ of
nonempty subsets of $[r]$ such that every element $i \in [r]$ belongs
to precisely one element of $\pi$ (i.~e., $[r]$ is a disjoint union
of the elements in $\pi$); $\#\pi$ denotes the number of elements
of $\pi$.
It is worth mentioning that the M\"obius inversion formula asserts
that \eqref{eq:DefCumulants} has an equivalent form:
        \begin{equation}
\label{eq:DefMoments}
            \bigoplus_{j \in J}X_j = \sum_{\substack{\pi \in \PPP(J) } } 
            \prod_{B \in \pi} \ka(X_i: i \in B).
    \end{equation}

Note that cumulants are multilinear. Thus, in order to understand
the discrepancy between $(\Lambda,\oplus)$ and $(\Lambda,\cdot)$ it is
enough to study cumulants $\ka(\tilde{H}_{\la_1},\dots,
\tilde{H}_{\la_r})$ of the basic elements.  It is clear that $\ka(\tilde{H}_{\la^1},\dots,
\tilde{H}_{\la^r}) \in \Z[q,t]\{m_\mu\}_\mu$ because the cumulant $\ka(\tilde{H}_{\la^1},\dots,
\tilde{H}_{\la^r})$ is a linear combination of products of Macdonald
polynomials. It is also clear that $\ka(\tilde{H}_{\la^1},\dots,
\tilde{H}_{\la^r}) \in \Z[q,t]\{(q-1)m_\mu\}_\mu$ for $r
>1$ since $(\Lambda,\oplus)$ and
$(\Lambda,\cdot)$ coincide in the specialization $q=1$, but one check that $\ka(\tilde{H}_{\la^1}, \tilde{H}_{\la^2},
\tilde{H}_{\la^3}) \in \Z[q,t]\{(q-1)^2m_\mu\}_\mu$, which is quite
nontrivial. We can observe a pattern here -- it is reasonable to think that ``higher--order'' cumulants provide
the approximation of higher order when $q \to 1$, that is $\ka(\tilde{H}_{\la^1},\dots,
\tilde{H}_{\la^r}) \in \Z[q,t]\{(q-1)^{r-1}m_\mu\}_\mu$. This statement was conjectured in
\cite{DolegaFeray2016}, and this implies the partial solution of the
$b$--conjecture, see \cref{subsub:motivations}. The proof was found recently
by the author:
\begin{theorem}{\cite{Dolega2017}}
\label{theo:podzielnosc}
For any partitions $\lambda^1,\dots,\lambda^r$ one has
\[ \ka(\tilde{H}_{\la^1},\dots,
\tilde{H}_{\la^r}) \in \Z[q,t]\{(q-1)^{r-1}m_\mu\}_\mu.\]
\end{theorem}

\subsubsection{Motivations}
\label{subsub:motivations}
The initial motivation for studying cumulants $\ka(\tilde{H}_{\la_1},\dots,
\tilde{H}_{\la_r})$ comes from our attempts \cite{DolegaFeray2016} to prove the $b$-conjecture
-- one of the major open problems in the theory of Jack symmetric
functions posed by Goulden and Jackson \cite{GouldenJackson1996}.
The $b$-conjecture states that the coefficients of a certain multivariate
generating function $\psi(\xx,\yy,\zz;\beta)$ involving Jack symmetric functions can be
interpreted as weighted generating functions of graphs embedded into
surfaces. Except some special cases
\cite{BrownJackson2007, LaCroix2009, KanunnikovVassilieva2016,
  Dolega2017a} not much is known and the $b$-conjecture is still wide
open. However, in our recent paper \cite{DolegaFeray2016} the author and F\'eray were able to
rewrite the function $\psi(\xx,\yy,\zz;\beta)$ as a linear combination
of cumulants of Jack symmetric functions, which are specializations of
cumulants $\ka(\tilde{H}_{\la_1},\dots,
\tilde{H}_{\la_r})$, and we showed that \cref{theo:podzielnosc}
implies a partial solution of the $b$-conjecture. In view of this result, understanding of the structure of cumulants $\ka(\tilde{H}_{\la_1},\dots,
\tilde{H}_{\la_r})$ is of great interest as a potential tool for solving
the $b$-conjecture in general.

Furthermore, cumulants $\ka(\tilde{H}_{\la_1},\dots,
\tilde{H}_{\la_r})$ are of a special interest from the following
reason: the structure of Macdonald polynomials $\tilde{H}_\mu$ is
directly related to algebraic geometry \cite{Haiman2002}. It turns out that
cumulants appear naturally in algebraic geometry \cite{DiNardoWynnZwiernik2017} and it is
interesting to investigate what kind of geometric information is encoded
in the structure of $\ka(\tilde{H}_{\la_1},\dots,
\tilde{H}_{\la_r})$. Secondly, we recall that one of the most typical
application of cumulants in the context of probability is to show that a certain family of random
variables is asymptotically Gaussian. Especially, when one deals with
discrete structures, whose ``observables'' form a nice algebraic structure, the main
technique is to show that conditional cumulants have a certain \emph{small cumulant
  property} exactly of the same form as in \cref{theo:podzielnosc}; see \cite{Sniady2006c, FerayMeliot2012, Feray2013,
  DolegaSniady2018}. It is therefore natural to ask for a
probabilistic interpretation of \cref{theo:podzielnosc}, which leads
to some kind of a central limit theorem. The most natural framework
to investigate this problem seems to be related to Macdonald
processes introduced by Borodin and Corwin \cite{BorodinCorwin2014}
and it would be interesting to link our work with this probabilistic
aspect.

Finally, the biggest motivation for us to study cumulants $\ka(\tilde{H}_{\la_1},\dots,
\tilde{H}_{\la_r})$ is their beautiful and mysterious
combinatorial structure. In particular, their Schur--positivity is yet
to be resolved. For partitions $\la^1,\dots,\la^r$
we define the \emph{Macdonald cumulant}
$\ka(\la^1,\dots,\la^r)(\xx;q,t)$ as
\begin{equation}
\label{eq:MacdoCumu}
\ka(\la^1,\dots,\la^r)(\xx;q,t) := \frac{\ka(\tilde{H}_{\la_1}(\xx;q,t),\dots,
\tilde{H}_{\la_r}(\xx;q,t))}{(q-1)^{r-1}}.
\end{equation}
We recall that monomial symmetric functions have integer coefficients in the Schur
basis expansion. Thus, one can reformulate \cref{theo:podzielnosc} as
follows: for any partitions $\la^1,\dots,\la^r$ one has the following expansion
\[ \ka(\la^1,\dots,\la^r) \in \Z[q,t]\{s_\mu\}_\mu.\]
Remarkably, extensive computer simulations suggest that Macdonald
cumulants are, in fact, Schur--positive, which we conjectured in our
recent paper \cite{Dolega2017}:

\begin{conjecture}[Higher--order Macdonald positivity conjecture]
\label{conj:GeneralizedKostka}
Let $\lambda^1,\dots,\lambda^r$ be partitions. Then, for any
partition $\mu$, the \emph{multivariate $q,t$-Kostka number}
$\tilde{K}^{(q,t)}_{\mu; \la^1,\dots,\la^r}$ defined by the following
expansion
\[ \ka(\la^1,\dots,\la^r) :=
  \sum_{\mu}\tilde{K}^{(q,t)}_{\mu; \la^1,\dots,\la^r}\ s_\mu\]  
is a polynomial in $q,t$ with
\textbf{nonnegative integer} coefficients.
\end{conjecture}

Note that the case $r=1$ corresponds to the Macdonald positivity
ex-conjecture, so our conjecture generalizes it from the cumulant of
order $1$ to cumulants of higher order. 


\subsection{The main results}

Our main result is an explicit combinatorial formula for Macdonald
cumulants $\ka(\la^1,\dots,\la^r)$. Before we go into details of the formula, let us briefly summarize its consequences. 

First of all,
our main result strengthens \cref{theo:podzielnosc} twofold. On the one
hand -- \cref{theo:podzielnosc} is an immediate consequence of
our explicit formula, while the original proof relied on some
complicated induction and was not constructive; in particular it
asserted that the coefficients of the monomial expansion of Macdonald
cumulants belong to $\Z[q,t]$ by some abstract argument. On the
other hand our
formula shows that Macdonald
cumulants are monomial--positive and the coefficients in this
expansion have an explicit combinatorial interpretation in terms of
counting trees with some weights. 

Secondly, we deduce from our formula an explicit, $q,t$--positive expansion of
Macdonald cumulants in fundamental quasisymmetric
functions.

Finally, we would like to comment a relation between our main result
and \cref{conj:GeneralizedKostka}. There is a well-known combinatorial formula expanding Schur
polynomials as a linear combination of monomial symmetric functions
with nonnegative integer coefficients. One can invert this formula to
expand monomial symmetric function in Schur basis and the coefficients
in this expansion
are no longer positive in general. In particular,
\cref{conj:GeneralizedKostka} stronger than positivity and
integrality of Macdonald
cumulants in monomial basis. However, it turned out that our combinatorial formula implies that \cref{conj:GeneralizedKostka} holds true in the
special case of hooks. In other words, for any partitions
$\la^1,\dots,\la^r$ and for any partition $\mu$ of the hook shape
(of the form $\mu = (r+1,1^s)$ for some nonnegative integers
$r,s$) the multivariate $q,t$-Kostka number
$\widetilde{K}_{(r+1,1^s); \la^1,\dots,\la^r}(q,t)$ is a polynomial in $q,t$ with
nonnegative integer coefficients. Here, as before, we interpret the
polynomial $\widetilde{K}_{(r+1,1^s); \la^1,\dots,\la^r}(q,t)$ as a
generating series of some trees. There is a hope that our
combinatorial formula can be transformed into a combinatorial proof of
the Schur--positivity of Macdonald cumulants in the future (see
\cref{sec:open} for more details), but so far this is a big open
problem even in the case of Macdonald polynomials.

\subsubsection{Graphs and our main theorem}
Let us introduce the graph theory terminology necessary for
presenting our
main result. Let $G = (V,E)$ be a
connected multigraph, possibly with loops, where $V = [r]$. The vertex
with label $1$ is called the \emph{root}. For
any vertices $i,j \in V$ let $e_{i,j}(G)$ denote the number of edges
linking $i$ with $j$ in $G$. We say that
$H \subset G$ is a \emph{spanning subgraph} of $G$, if for any
vertex $v \in V$ there exists an edge in $H$ containing $v$. We say
that $T \subset G$ is a \emph{spanning tree} of $G$ if it is a
spanning subgraph of $G$ and it is a tree
(it is connected and has no cycles). For a pair of different
vertices $i,j \in V$ of $T$ we say that $j$ is a \emph{descendant} of $i$ if
$i$ lies on the shortest path from $j$ to the root, and we call $i$ an
\emph{ancestor} of $j$. If $i$ is an ancestor of $j$ adjacent to it,
we call it a \emph{parent} of $j$.
We say that a pair $(i,j)$ which does not contain the root is a \emph{$\ka$-inversion} of a spanning
tree $T$ of $G$ if it is an \emph{inversion} ($i$ is an ancestor of
$j$ and $i>j$) and $j$ is adjacent to the
parent of $i$ in $G$.
Let $\tilde{G}$ be a graph obtained
from $G$ by replacing all multiple edges by single ones.
We define the \emph{$G$-inversion polynomial} by
\begin{equation}
\label{eq:GesselSagan'}
\I_G(q) = q^{\text{number of loops in } G}\sum_{T \subset \tilde{G}}
q^{\ka(T)}\prod_{\{i,j\} \in T}[e_{i,j}(G)]_q,
\end{equation}
where the sum runs over all spanning trees of $\tilde{G}$, 
\begin{equation}
\label{eq:ka(T)}
\ka(T) = \sum_{\{i,j\} - \ka-\text{inversion in }
  T}e_{\text{parent}(i),j}(G),
\end{equation}
and we use a standard notation $[n]_q := \frac{q^n-1}{q-1} =
1+q+\cdots+q^{n-1}$.
For example, if we want to compute $\I_G(q)$ for $G$ from
\cref{fig:Przyklad}, we first notice that there are 3 spanning trees
of $\tilde{G}$ which we denote by $T_1,T_2$ and $T_3$.
\begin{figure}[h!]
\vspace{-10pt}
\label{subfig:A3}
    \begin{tikzpicture}[scale=0.3, white/.style={circle,draw=black,fill=white,inner sep=3pt},
	blue/.style={circle,draw=black,fill=cyan!70,inner sep=3pt}, red/.style={circle,draw=black,fill=BrickRed,inner sep=3pt}]

\begin{scope}
		\draw (3,0) node (W) [white] {$1$};
		\draw (0,5) node (B) [blue] {$2$};
		\draw (6,5) node (R) [red] {$3$};

                \draw[gray, thick] (W) to node[anchor=east] {$6$} (B);
                \draw[gray, thick] (W) to node[anchor=west] {$6$} (R);
                \draw[gray, thick] (B) to node[anchor=south] {$8$} (R);
                \draw[gray, thick] (B) to 
                [out=120,in=160,looseness=4] (B);
                \draw[gray, thick] (B) to
                [out=100,in=180,looseness=5] (B);
                \draw[gray, thick] (R) to
                [out=0,in=80,looseness=5] (R);
                \draw[gray, thick] (W) to
                [out=310,in=230,looseness=5] (W);
\draw (3,-4) node (X) {$\tilde{G}$};
\end{scope}
    \begin{scope}[xshift=12cm]

		\draw (3,0) node (W) [white] {$1$};
		\draw (0,5) node (B) [blue] {$2$};
		\draw (6,5) node (R) [red] {$3$};

                \draw[black, double, very thick] (W) to node[anchor=east] {$6$} (B);
                \draw[gray, thick] (W) to node[anchor=west] {$6$} (R);
                \draw[black, double, very thick] (B) to node[anchor=south] {$8$} (R);
                \draw[gray, thick] (B) to 
                [out=120,in=160,looseness=4] (B);
                \draw[gray, thick] (B) to
                [out=100,in=180,looseness=5] (B);
                \draw[gray, thick] (R) to
                [out=0,in=80,looseness=5] (R);
                \draw[gray, thick] (W) to [out=310,in=230,looseness=5]
                (W);
\draw (3,-4) node (X) {$T_1$};
\end{scope}
    \begin{scope}[xshift=24cm]

		\draw (3,0) node (W) [white] {$1$};
		\draw (0,5) node (B) [blue] {$2$};
		\draw (6,5) node (R) [red] {$3$};

                \draw[black, double, very thick] (W) to node[anchor=east] {$6$} (B);
                \draw[black, double, very thick] (W) to node[anchor=west] {$6$} (R);
                \draw[gray, thick] (B) to node[anchor=south] {$8$} (R);
                \draw[gray, thick] (B) to 
                [out=120,in=160,looseness=4] (B);
                \draw[gray, thick] (B) to
                [out=100,in=180,looseness=5] (B);
                \draw[gray, thick] (R) to
                [out=0,in=80,looseness=5] (R);
                \draw[gray, thick] (W) to
                [out=310,in=230,looseness=5] (W);
\draw (3,-4) node (X) {$T_2$};
\end{scope}
    \begin{scope}[xshift=36cm]
		\draw (3,0) node (W) [white] {$1$};
		\draw (0,5) node (B) [blue] {$2$};
		\draw (6,5) node (R) [red] {$3$};

                \draw[gray, thick] (W) to node[anchor=east] {$6$} (B);
                \draw[black, double, very thick] (W) to node[anchor=west] {$6$} (R);
                \draw[black, double, very thick] (B) to node[anchor=south] {$8$} (R);
                \draw[gray, thick] (B) to 
                [out=120,in=160,looseness=4] (B);
                \draw[gray, thick] (B) to
                [out=100,in=180,looseness=5] (B);
                \draw[gray, thick] (R) to
                [out=0,in=80,looseness=5] (R);
                \draw[gray, thick] (W) to [out=310,in=230,looseness=5]
                (W);
\draw (3,-4) node (X) {$T_3$};
\end{scope}
\end{tikzpicture}
        \caption{$G = ([3],E)$, where $e_{1,2}(G) = e_{1,3}(G) = 6$
          and $e_{2,3}(G) = 8$. $T_1,T_2$ and $T_3$ are indicated by
          double edges.}
        \label{fig:Przyklad}
    \end{figure}
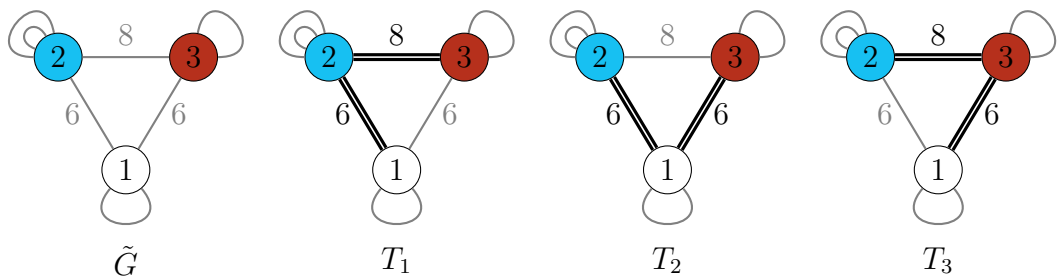
A pair $(3,2)$
is not a $\ka$--inversion of $T_1$ nor $T_2$, but it is a
$\ka$--inversion of $T_3$. Therefore $\ka(T_1) = \ka(T_2) = 0$ and
$\ka(T_3) = e_{\text{parent}(3),2}(G) = e_{1,2}(G) = 6$. Thus, 
\begin{multline*}
\I_G(q) = q^4\left(q^0[6]_q[8]_q+q^0[6]_q[6]_q+q^6[6]_q[8]_q\right)
  =\\
q^4(1+q+q^2+q^3+q^5)(2+2q+2q^2+2q^3+2q^4+2q^5+2q^6+2q^7+q^8+q^9+q^{10}+q^{11}+q^{12}+q^{13}).
\end{multline*}
Note that the $G$--inversion polynomial for a complete graph $G = K_r$
is given by
\[ \I_G(q) = \sum_{T} q^{\inv(T)}.\]
Thus, this is an \emph{inversion polynomial} -- a polynomial which counts the
labeled trees on $r$ vertices with respect to the number
of their inversions. The $G$-inversion polynomial has various
interpretations in terms of \emph{Tutte polynomials, $G$-parking
functions, the abelian sandpile model} or \emph{Tesler matrices} among
others -- see \cref{sec:TutteParkinETC,sec:open} for more details.

We are ready to formulate our main result. 
\begin{theorem}
\label{theo:MacdonaldCumuFormula}
Let $\la^1,\dots,\la^r$ be partitions. Then, the following formula
holds true:
\begin{equation}
\label{eq:MacdonaldCumuFormula}
\ka(\la^1,\dots,\la^r) = \sum_{\si: \la^{[r]} \to
  \N_+} \I_{G^\si_{\la^1,\dots,\la^r}}(q)\ t^{\maj(\sigma)}\ \xx^\sigma,
\end{equation}
where $\la^B := \bigoplus_{b \in B}\la^b$ and $\xx^\sigma := \prod_{\square \in \la^{[r]}}x_{\sigma(\square)}$.
\end{theorem}

The summation index in \eqref{eq:MacdonaldCumuFormula} runs over all
the fillings $\sigma$ of a Young diagram $\la^{[r]}$ by positive
integers, $G^{\sigma}_{\la^1,\dots,\la^r}$ is a certain multigraph
(we refer for its construction to \eqref{subsect:Identification}), and
$\maj(\sigma)$ is a certain statistic of the filling $\sigma$ (see
\eqref{subsec:Fillings} for the precise definition). We finish this
section by mentioning that our formula
\eqref{eq:MacdonaldCumuFormula} specializes to the celebrated formula
of Haglund, Haiman and Loehr \eqref{eq:Haglund} when $r=1$.
The work of Haglund, Haiman and Loehr \cite{HaglundHaimanLoehr2005},
where many consequences of the formula \eqref{eq:Haglund}
were presented, was a source of inspiration for our research and most
of the consequences of our formula \eqref{eq:MacdonaldCumuFormula} can
be derived in a similar manner as in \cite{HaglundHaimanLoehr2005}.

\subsection{Organization of the paper}
\label{subsec:outline}

In \cref{sec:TutteParkinETC} we discuss various interpretations of the
polynomial $\I_G(q)$ and we prove some of its
properties. \cref{sec:preliminaries} introduces the necessary
background on the combinatorics of Macdonald polynomials. \cref{sec:Coloring} is devoted to the proof
of our main result and explains the
construction of graphs involved in our formula. \cref{sec:Fundamental}
gives an explicit formula for the fundamental quasisymmetric functions
expansion of Macdonald cumulants. \cref{sec:Kostka} is devoted to the
proof of the Schur--positivity of Macdonald cumulants in the case of hooks.
In \cref{sec:FullyColored} we investigate a certain subfamily of
Macdonald cumulants which arises from the study of the $b$-conjecture and we show
that this family is a basis of the
symmetric functions algebra. We finish with \cref{sec:open}, where we
state some open problems and possible directions for the future research.

\section{Tutte
  polynomials, $G$-inversion polynomials and $G$-parking functions}
\label{sec:TutteParkinETC}

\subsection{Tutte polynomials and $G$-inversion polynomials}
\label{subsec:Tutte}

From now on each graph $G = (V,E)$ can possibly have multiple edges and
loops. Moreover, in this section we additionally assume that $G$ is connected. The
\emph{Tutte polynomial} of $G$ denoted by $\Tu_G(x,y)$ was introduced
by Tutte in \cite{Tutte1947}, and its various specializations give rise to many important graph invariants
such as the number of its spanning trees or the number of its acyclic
orientations, among others. This remarkable interdisciplinary of $\Tu_G(x,y)$ made it one
of the most important invariants in modern graph theory, see \cite{Bollobas1998}.


Tutte polynomials are defined by the following equality:
\begin{equation}
\label{eq:DefTutte}
\Tu_G(x,y) = \sum_{H \subset G}(x-1)^{c(H)-1}\ (y-1)^{\#E(H)-\#V+c(H)},
\end{equation}
where we sum over all (possibly disconnected) sub-graphs of $G$, $c(H)$ denotes the
number of connected components of $H$, and $E(H)$ is the set of edges
of $H$. Tutte \cite{Tutte1954} noticed that for a connected graph $G$ the specialization
$\Tu_G(1,1)$ counts the number of \emph{spanning trees} of $G$,
and this
observation allows him to express $\Tu_G(x,y)$ as the bivariate
generating function of spanning trees of $G$:
\[ \Tu_G(x,y) = \sum_{T \subset G}x^{\ia(T)}\ y^{\ea(T)} \in \N[x,y],\]
where $\ia(T)$, and $\ea(T)$ are certain statistics of
a spanning tree $T$, called \emph{internal} and \emph{external activities}. 

In this paper we will be entirely focused on the specialization
$\Tu_G(1,q)$, which is of a special
interest as it appears naturally in many different contexts. Gessel, and
Gessel with Sagan
\cite{Gessel1995,GesselSagan1996} interpreted $\Tu_G(1,q)$
as the generating
function of spanning trees of $G$ with respect to the so-called
\emph{$\ka$-inversion} statistic, which is more natural and simpler
than the external activity.
We label vertices of $G$ in an arbitrary way by consecutive nonnegative integers. We also set
for every pair of distinct vertices of $G$ an arbitrary linear order
on the set of edges linking this pair of vertices and for any edge $e
\in E(G)$ we define $s(e)$ as the number of edges in $G$ strictly greater
then $e$. We extend this definition to any subgraph $H \subset G$
by setting $s(H) := \sum_{e \in H}s(e)$. We recall that an
inversion $i > j$ in a tree $T \subset G$ forms a
\emph{$\ka$-inversion} if the parent of $i$ is
adjacent to $j$ in the graph $G$. It was shown
by Gessel and Sagan \cite{GesselSagan1996} that
\[ \Tu_G(1,q) = q^{\text{number of loops in} G}\ \sum_{T \subset G} q^{\ka(T)+s(T)},\]
where we sum over all spanning trees of $G$, and $\ka(T)$ is given by \eqref{eq:ka(T)}.
It is easy to show that above formula can be rewritten in the form
\eqref{eq:GesselSagan'}, thus
\begin{equation} 
\label{eq:wow}
\I_G(q) = \Tu_G(1,q).
\end{equation}
In particular, the polynomial $\I_G(q)$ depends only on the structure
of $G$ and is invariant under permuting the labels of the vertices.
When $G$ is a graph with no multiple edges nor loops, Gessel \cite{Gessel1995} found that
\[ \I_G(q) = \sum_{\substack{T \subset G\\ \ka(T) = 0}}\prod_{w \in V\setminus\{v\}}[\delta_T(w)]_q,\]
where $v$ is the root of $G$, $\delta_T(w)$ is the number of descendants of $w$ (including
$w$) adjacent to the parent of $w$ in $G$.
In fact, the same argument as used by Gessel allows us to extend his formula to the
general case of graphs (with multiple edges and loops), which will be useful for us later. We
recall that a graph $\tilde{G}$ is obtained from $G$ by replacing all multiple edges by
single ones.

\begin{proposition}
\label{prop:TutteAnother}
Let $G$ be a graph. Then
\begin{equation}
\label{eq:TutteAnother}
\I_G(q) = q^{\text{number of loops in }G}\sum_{\substack{T \subset
    \tilde{G}\\ \ka(T) = 0}}\prod_{w \in V\setminus\{v\}}[\delta_T(w)]_q,
\end{equation}
with
\[ \delta_T(w) = \sum_{i}e_{i,\text{parent}(w)}(G),\]
where we sum over all descendants of $w$ (including
$w$), and $v$ is the root of $G$.
\end{proposition}

\begin{proof}
Using \eqref{eq:DefTutte} we obtain the formula for the specialization
\begin{equation} 
\label{eq:connect}
c_G(q) := q^{\#V-1}\Tu_G(1,q+1) = \sum_{H \subset G} q^{\#E(H)},
\end{equation}
where we sum over all connected sub-graphs of $G$. Let $U \subset
V$ be a non-empty subset of vertices of $G$, and for $w \in V$ we
define $d_G(U,w)$ as the number of edges in $G$ connecting $w$ with
some vertex from $U$. Note that erasing a vertex $w$ from the
connected graph $H$ splits this graph into a collection of connected sub-graphs
$H_1,\dots,H_l$ with the corresponding sets of their vertices
$V_1,\dots,V_l$. Then $\{V_1,\dots,V_l\} \in \PPP(V\setminus\{w\})$ and for all $1 \leq i \leq l$ there is at
least one edge linking $w$ with $V_i$. This leads to the following
recursion
\[ c_G(q) = q^{\text{number of loops in }w}\sum_{\pi \in
    \PPP(V\setminus\{w\})}\prod_{B \in
    \pi}\left((q+1)^{d_G(B,w)}-1\right)\ c_B(q), \]
which can be rewritten as
\begin{equation}
\label{eq:recursion}
\Tu_G(1,q) = q^{\text{number of loops in }w}\sum_{\pi \in
    \PPP(V\setminus\{w\})}\prod_{B \in
    \pi}[d_G(B,w)]_q\ \Tu_B(1,q)
\end{equation}
and it holds true for any vertex $w \in V$. It is therefore enough to show
that the right hand side of \eqref{eq:TutteAnother} satisfies the same recursion for $w=v$
being the root
(which implies that also for any other vertex).

Let $T$ be a spanning
tree of $\tilde{G}$ with $\ka(T) = 0$. If we delete its root $v$, we obtain
a collection of trees $T_1,\dots,T_l$ with the corresponding sets of their vertices
$V_1,\dots,V_l$, and their roots $v_1,\dots,v_l$. Then, for each $1 \leq i \leq l$ the graph $T_i$ is a
spanning tree of $G|_{V_i}$ and $\ka(T_i) = 0$. Thus, for any such
a tree $T$ one has
\begin{multline*} 
q^{\text{number of loops in }G}\prod_{w \in V\setminus\{v\}}[\delta_T(w)]_q
\\
= q^{\text{number of loops in }v}\prod_{1 \leq
    i \leq l}\left(\left[\sum_{w \in V_i}e_{v,w}(G)\right]_q\ q^{\text{number of
        loops in }G_i}\ \prod_{w \in
      V_i\setminus\{v_i\}}[\delta_{T_i}(w)]_q\right).
\end{multline*}
Since $\sum_{w \in V_i}e_{v,w}(G) = d_G(V_i,v)$, the right hand side
of \eqref{eq:TutteAnother} satisfies recursion \eqref{eq:recursion},
which finishes the proof of \cref{prop:TutteAnother}.
\end{proof}

We recall that a rooted tree $T$ is \emph{increasing} if it contains no inversions.

\begin{corollary}
\label{cor:specjalny}
Let $a_2,\dots,a_r$ be positive integers, and let
$G_{a_2,\dots,a_r} := (V,E)$ be a graph without loops such that
$e_{i,j}(G) = a_{\max(i,j)}$ for each $i\neq j \in V = [r]$. Then
\begin{equation} 
\label{eq:PolynomialP}
\I_{G_{a_2,\dots,a_r}}(q) = P_{a_2,\dots,a_r}(q) := \sum_{T \text{ increasing tree on  }[r]
  }\ \prod_{2 \leq i \leq r} [\delta_T(i)]_q,
\end{equation}
and $\delta_T(i) := \sum_j a_j$, where $j$ ranges over descendants of
$i$ (including $i$ itself ). 
\end{corollary}

\begin{proof}
\cref{prop:TutteAnother} asserts the following
formula
\[ \I_{G_{a_2,\dots,a_r}}(q) = \sum_{\substack{T \subset
    \widetilde{G_{a_2,\dots,a_r}}\\ \ka(T) = 0}}\ \prod_{2 \leq i
  \leq r} [\delta_T(i)]_q,\]
where $\delta_T(i) = \sum_j a_{\max(\text{parent}(i),j)}$. Note that $\widetilde{G_{a_2,\dots,a_r}}$ is the
complete graph $K_r$, therefore its spanning trees $T$ with $\ka(T) =
0$ are precisely increasing trees. Thus, for any descendant of $i$ one has
$\max(\text{parent}(i),j) = j$, which
finishes the proof.
\end{proof}

\begin{remark}
In this section we assumed that a graph $G$ is connected. Typically when $G$ is not connected the
corresponding Tutte polynomial is defined as the product of Tutte
polynomials of each connected component.
However, for our purposes we extend the definition of the Tutte
polynomial to non-connected graphs by setting its value to
$0$, which agrees with the idea that this is a weighted generating function
of spanning trees of $G$.
\end{remark}

We finish this section by an important lemma which links Tutte polynomials with cumulants. This lemma can
be also found in \cite[Proposition 4.1]{Josuat-Verges2013}, but our
proof differs from the one of Josuat-Verges.

\begin{lemma}
\label{lem:TutteCumulant}
Let $G = (V,E)$ be a graph. Then
\[ \Tu_G(1,q) = (q-1)^{1-\#V} \sum_{\substack{\pi \in \PPP(V) } } 
            (-1)^{\#\pi-1}(\#\pi-1)!\ \prod_{B \in \pi} q^{\#E|_B},\]
where $E|_B$ denotes the subset of $E$ consisting of the edges with
both endpoints from the set $B \subset V$.
\end{lemma}

\begin{proof}
Let $G$ be a graph (possibly disconnected), and we define two
generating functions
\[ nc_G(q) = \sum_{H \subset G} q^{\#E(H)},\]
where we sum over all (possibly disconnected) sub-graphs of $G$ and 
\[ c_G(q) = \sum_{H \subset G} q^{\#E(H)},\]
where we sum over all connected sub-graphs of $G$.
Then, clearly
\[ nc_G(q) = \sum_{\pi \in \PPP(V)}\prod_{B \in \pi}c_{G|_B}(q).\]
Thus, the M\"obius inversion formula
(\eqref{eq:DefCumulants}--\eqref{eq:DefMoments}) implies that
\[ c_G(q) = \sum_{\pi \in \PPP(V)}(-1)^{\#\pi-1}(\#\pi-1)!\ \prod_{B
    \in \pi}nc_{G|_B}(q).\]
Plugging $nc_{G|_B}(q) = (1+q)^{\#E|_B}$ and \eqref{eq:connect} into
the above equality yields the desired result.
\end{proof}

\subsection{$G$-parking functions and the abelian sandpile model}
\label{subsec:GParking}

The polynomial $\I_G(q)$ is also a generating function of two other objects
of interest: $G$-parking
functions, and recurrent
configurations in an abelian sandpile model on $G$.

Let $G = (V,E)$ be a graph
with the set of vertices $V = [r]$, where $r\geq 1$ is a positive
integer and we denote the root of $G$ by $v \in [r]$. For
any $i \in U \subset [r]\setminus\{v\}$ we define the \emph{outdegree}
$\outdeg_U(i)$ of a vertex $i$ as the number of edges in $G$ linking $i$
with some vertex $j \notin U$. We call a function $f: [r]\setminus\{v\} \to \N$ a \emph{$G$-parking function} if for any nonempty subset $U \subset [r]\setminus\{v\}$ there exists $i \in U$
such that $f(i) < \outdeg_U(i)$. For example, when $G = K_{r}$ is the complete graph on $[r]$, then the set of $G$-parking functions is
precisely the set of parking functions. 

Postnikov and Shapiro noticed that $G$-parking functions are directly
related to \emph{recurrent configurations}
in the \emph{abelian sandpile model for $G$}, which is a model where we are
trying to distribute chips among vertices of our graph. A function $u : [r]\setminus\{v\} \to
\N$ giving the number of chips placed in vertices of $G$ different
from the root is called a \emph{configuration}. We say that a vertex $i \in [r]\setminus\{v\}$ is \emph{unstable} if $u(i)
\geq \deg(i)$ -- if this is a case, this vertex can \emph{topple} by
sending chips to adjacent vertices one along each incident
edge. We say that a configuration is \emph{stable} if all the
vertices $i \in [r]\setminus\{v\}$ except the root are stable. For the root we set
$u(v) = -\sum_{i \in [r]\setminus\{v\}}u(i)$, and the root can always topple. Finally, we say that a
configuration $u$ is \emph{recurrent} if there exists a nontrivial
configuration $u' \neq 0$ such that $u$ can be obtained from $u+u'$ by a sequence
of topplings. Postnikov and Shapiro noticed that a configuration $u$ is
recurrent if and only if $f:[r]\setminus\{v\}\to\N$ defined by $f(i):=
\deg(i)-u(i)-1$ is a $G$-parking function. 
We define a \emph{weight} of a $G$-parking function $f$:
\[ \wt(f) := \#E-(r-1)-\sum_{i \in [r]\setminus\{v\}}f(i),\]
and we define a $q$-generating function of $G$-parking functions $P_G(q) := \sum_{f}q^{\wt(f)}$ with
respect to their weights. We can also interpret $P_G(q)$ as the generating function of recurrent
configurations on $G$ with respect to their \emph{level}, where
\[ \level(u) := \sum_{i \in [r]\setminus\{v\}}u(i)+\deg(0)-\#E.\]

Merino L\'opez proved \cite{MerinoLopez1997} that $P_G(q) = \Tu_G(1,q)
= \I_G(q)$, therefore we have two additional interpretations of the $G$-inversion polynomial
$\I(q)$.

\section{Preliminaries on symmetric functions and Young diagrams}
\label{sec:preliminaries}

\subsection{Partitions and Young diagrams}
\label{SubsecPartitions}
We call $\la := (\la_1,\la_2,\ldots,\la_l)$ \emph{a composition} of $n$
if it is a sequence of nonnegative
integers such that $\la_1+\la_2+\cdots+\la_l = n$ and $\la_l > 0$. If $\la$
is a weakly decreasing sequence we call it \emph{a partition} of $n$.
Then $n$ is called {\em the size} of $\lambda$ and $l$ is {\em its length}.
We use the notation $\la \vdash n$, or $|\la| = n$ to indicate its
size, and $\ell(\la) = l$ for its length.

There exists a canonical involution on the set of partitions which associates with a partition $\la$ its \emph{conjugate partition} $\la^t$.
By definition, the $j$-th part $\la_j^t$ of the conjugate partition
is the number of positive integers $i$ such that $\la_i \ge j$.
We define the partial order called \emph{dominance order} on the set
 of partitions of the same size as follows:
\[ \la \geq \mu \iff \sum_{1 \leq i \leq j} \la_i  \geq \sum_{1 \leq i\leq j} \mu_i \text{ for any positive integer } j.\]

\vspace{10pt}

A partition $\la$ is identified with some geometric object, called \emph{Young diagram},
defined by:
\[ \la = \{(i, j):1 \leq i \leq \la_j, 1 \leq j \leq \ell(\la) \}.\]
 For any
 box $\square := (i,j) \in \la$ from a Young diagram we define its
 \emph{arm-length} by $a_\la(\square) := \la_j-i$, its
 \emph{coarm-length} by $a'_\la(\square) := i-1$, its
 \emph{leg-length} by $\ell_\la(\square) := \la_i^t-j$ and its \emph{coleg-length} by $\ell_\la'(\square) := j-1$ (the same definitions as in \cite[Chapter I]{Macdonald1995}), see \cref{fig:ArmLeg}.

    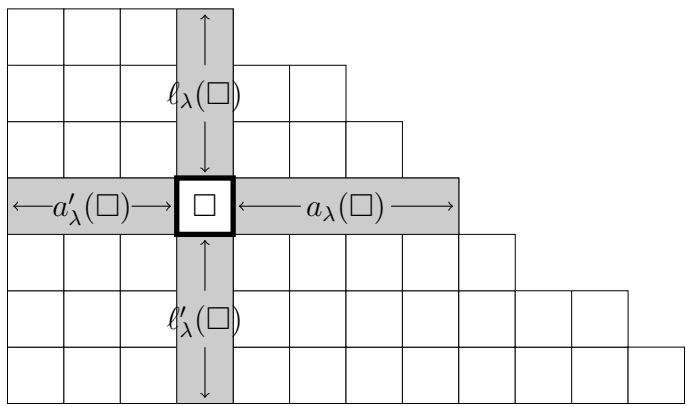
\begin{figure}[t]
        \[
    \begin{array}{c}
    \Yfrench
    \Yboxdim{1cm}
    \begin{tikzpicture}[scale=0.75]
        \newcommand\yg{\Yfillcolour{gray!40}}
        \newcommand\yw{\Yfillcolour{white}}
        \newcommand\tralala{(i,j)}
        \newcommand\arml{a_\la(\square)}
        \newcommand\legl{\ell_\la(\square)}
        \newcommand\armlp{a_\la'(\square)}
        \newcommand\leglp{\ell_\la'(\square)}
        \tgyoung(0cm,0cm,;;;|3;;;;;;;;,;;;:;;;;;;;,;;;:;;;;;,_3;_4,;;;|3;;;,;;;:;;,;;;:)
       \Yfillopacity{0}
       \Ylinethick{0.3pt}
       \Yfillopacity{1}
       \Yfillcolour{gray!40}
       \tgyoung(4cm,3cm,_4\arml)
       \tgyoung(3cm,4cm,|3\legl)
       \tgyoung(0cm,3cm,_3\armlp)
       \tgyoung(3cm,0cm,|3\leglp)
       \Yfillopacity{0}
       \Ylinethick{2pt}
       \tgyoung(3cm,3cm,\square)
       \draw[->](3.5,5)--(3.5,4.1);
       \draw[->](3.5,2)--(3.5,2.9);
       \draw[->](3.5,1)--(3.5,0.1);
       \draw[->](3.5,6)--(3.5,6.9);
       \draw[->](0.8,3.5)--(0.1,3.5);
       \draw[->](2.2,3.5)--(2.9,3.5);
       \draw[->](5.2,3.5)--(4.1,3.5);
       \draw[->](6.8,3.5)--(7.9,3.5);
       \end{tikzpicture}
    \end{array}
        \]
     \caption{Arm-, coarm-, leg- and coleg-lengths of a box in a Young diagram.}
     \label{fig:ArmLeg}
    \end{figure}

\subsection{Macdonald polynomials and plethysm}
\label{subsect:Macdonald}

Let $p_i(\xx)$ be the \emph{power-sum} symmetric function, that is
\[ p_i(\xx) := \sum_{j \geq 1} x_j^i.\]
For any formal power series $A$ in the indeterminates $q,t,\xx$, we define
the \emph{plethystic
substitution} $p_i[A]$ as the result of substituting $a^i$
for each
indeterminate $a$ appearing in $A$. We extend this definition to any
symmetric function $f \in \Lambda$ by expanding it in the power-sum basis, and
then applying the plethystic
substitution as above,~i.e.
\[ f[A] := \sum_{\la}c_\la\ p_\la[A],\] 
where $f(\xx) = \sum_\la c_\la\ p_\la(\xx)$ and $p_\la[A] := \prod_i p_{\la_i}[A]$. 

Similarly as above, we define a
symmetric function $\omega p_\lambda(\xx) :=
(-1)^{|\la|+\ell(\la)}p_\la(\xx)$ and we extend the action of
$\omega$ on $\Lambda$ by linearity. We make a convention that
a bolded capital letter denotes the sum of countably many
indeterminates indexed by positive integers, for example $\XX :=
x_1 + x_2 + \cdots$.
Note that if $f$ is a homogeneous symmetric function of degree $n$ then
\begin{equation} 
\label{eq:pleth}
f[-\XX] = (-1)^n \omega f(\xx).
\end{equation}

There exists a scalar product on $\Lambda$,
called the \emph{Hall scalar product} which is defined on the Schur
basis $\{s_\la\}_\la$ by making it the orthonormal basis.

It turned out that there exists a unique family
$\{\tilde{H}_\mu(\xx;q,t)\}_\mu$ of the symmetric functions, indexed by partitions, which fulfills the
following conditions:

\begin{enumerate}[label=(C\arabic*)]
\item
\label{item1} 
$\tilde{H}_\mu[\XX(q-1);q,t] \in \QQ(q,t)\{s_\la\}_{\la \leq \mu^t}$,
\item
\label{item2}
$\tilde{H}_\mu[\XX(t-1);q,t] \in \QQ(q,t)\{s_\la\}_{\la \leq \mu}$,
\item
\label{item3}
$\langle \tilde{H}_\mu, s_{(|\mu|)}\rangle = 1$.
\end{enumerate}
The
elements of the above family are called the \emph{Macdonald
  polynomials}, and their characterization by conditions \ref{item1}--\ref{item3}
is equivalent to the characterization proved by Macdonald
\cite{Macdonald1995} (Macdonald used a different normalization; for
the proof of this equivalence see \cite[Proposition 2.6]{Haiman1999}).

\subsection{Fillings of Young diagrams}
\label{subsec:Fillings}

For any partition $\la \vdash n$ let $\si:\lambda \to \N_+$ be a filling of the boxes of
the diagram $\la$ by positive integers. A \emph{descent} of $\si$ is a
pair of entries $\si(\square) > \si(\square')$ such that $\square$ lies immediately
above $\square'$, that is $\square' = (i, j)$ and $\square = (i, j+1)$ for
some positive integers $i,j$. We define the \emph{set of
  descents} as follows:
\[ \Des(\si) := \{\square \in \la: \si(\square) > \si(\square') \text{ is a descent} \}.\]
The \emph{major index} $\maj(\si)$ of a filling
$\si$ is defined as:
\begin{equation}
\label{eq:major}
\maj(\si) := \sum_{\square \in \Des(\si)}(\ell_\la(\square)+1).
\end{equation}

    \begin{figure}[t]
        \[
          \begin{array}{c}
        \YFrench
    \Yboxdim{1cm}
    \begin{tikzpicture}[scale=0.7]
        \newcommand\yb{\Yfillcolour{cyan!70}}
        \newcommand\yr{\Yfillcolour{BrickRed}}
        \newcommand\yw{\Yfillcolour{white}}
        \newcommand\yg{\Yfillcolour{gray!20}}
        \newcommand\eleven{11}
        \tgyoung(0cm,0cm,12<10>\eleven246<12><14>\eleven
        <13>,!\yg 24!\yw 31!\yg 8<10>7!\yw 899,!\yg 674<10>9!\yw 9!\yg
        <13><13>,!\yw 111!\yg \eleven!\yw 444,!\yg 934!\yw 9)
       \Ylinethick{1.5pt}
        \draw[gray, thick] (0.5,4.5) to (1.5,4.5);
        \draw[gray, thick] (0.5,4.5) to [bend left] (2.5,4.5);
        \draw[gray, thick] (0.5,3.5) to [bend left] (3.5,3.5);
        \draw[gray, thick] (1.5,3.5) to [bend left] (3.5,3.5);
        \draw[gray, thick] (2.5,3.5) to (3.5,3.5);
        \draw[gray, thick] (0.5,3.5) to [bend left] (3.5,3.5);
        \draw[gray, thick] (0.5,2.5) to [bend left] (2.5,2.5);
        \draw[gray, thick] (3.5,2.5) to (4.5,2.5);
        \draw[gray, thick] (3.5,2.5) to [bend left] (5.5,2.5);
        \draw[gray, thick] (1.5,1.5) to (2.5,1.5);
        \draw[gray, thick] (2.5,1.5) to (3.5,1.5);
        \draw[gray, thick] (4.5,1.5) to [bend left]  (6.5,1.5);
        \draw[gray, thick] (5.5,1.5) to (6.5,1.5);
        \draw[gray, thick] (5.5,1.5) to [bend left] (7.5,1.5);
        \draw[gray, thick] (5.5,1.5) to [bend left] (8.5,1.5);
        \draw[gray, thick] (5.5,1.5) to [bend left] (9.5,1.5);
        \draw[gray, thick] (2.5,0.5) to [bend left] (4.5,0.5);
        \draw[gray, thick] (2.5,0.5) to [bend left] (5.5,0.5);
        \draw[gray, thick] (2.5,0.5) to [bend left] (6.5,0.5);
        \draw[gray, thick] (3.5,0.5) to (4.5,0.5);
        \draw[gray, thick] (3.5,0.5) to [bend left] (5.5,0.5);
        \draw[gray, thick] (3.5,0.5) to [bend left] (6.5,0.5);
        \draw[gray, thick] (7.5,0.5) to [bend left] (9.5,0.5);
        \draw[gray, thick] (8.5,0.5) to (9.5,0.5);
        \draw[gray, thick] (8.5,0.5) to [bend left] (10.5,0.5);

       \end{tikzpicture}
\end{array}
        \]
        \caption{Inversion triples in the above filling $\si$ are indicated by
          gray lines, while the set of descents is highlighted in light
          gray.}
        \label{fig:LegArm}
    \end{figure}

The second statistic that is of great importance in this paper is
a certain generalization of inversions in a permutation. First, we say
that two boxes $\square,\square' \in \la$ \emph{attack each other} if
either
\begin{itemize}
\item
they are in the same row: $\square = (i,j), \square' = (k,j)$, or;
\item
they are in consecutive rows, with the box in the upper row strictly to the right of
the one in the lower row: $\square = (i, j+1), \square' = (k, j)$, where $i > k$.
\end{itemize} 

The \emph{reading order} is the linear ordering of the entries of $\la$ given by reading them row by row,
top to bottom, and left to right within each row. We associate to a filling $\si$  its \emph{reading word} $w_\si$ by reading
its entries in the reading order. An \emph{inversion} of $\si$ is a pair of entries
$\si(\square) > \si(\square') $, where $\square,\square'$ attack each
other, and $\square$ precedes $\square'$ in the reading order.

We say that the ordered triple of boxes
$\square_1,\square_2,\square_3$ is \emph{counterclockwise
  increasing}, if one of the following conditions holds
true:
\begin{itemize}
\item
$\si(\square_1) \leq \si(\square_3) < \si(\square_2)$, or
\item
$\si(\square_3) < \si(\square_2) < \si(\square_1)$, or
\item
$\si(\square_2) < \si(\square_1) \leq \si(\square_3)$.
\end{itemize} 
We define the \emph{inversion triple} as a pair of boxes $(\square_1,\square_2)$, where $\square_1$ is a box lying in the same
row as $\square_2$ to its left and such that a triple $\square_1,\square_2,\square_3$ is  \emph{counterclockwise
  increasing}, where $\square_3$ is a box lying directly below $\square_1$. Here, the convention is that for $\square_1,\square_2$ lying in the
first row $\si(\square_3) < \min_{\square \in \la}{\si(\square)}$.
The set of inversion triples of $\si$ is denoted by
$\InvP(\si)$. \cref{fig:LegArm} presents an example of above
defined objects.

\begin{remark}
Note that an inversion triple $(\square_1,\square_2)$ is defined a
priori as \emph{a
pair} of boxes, not as \emph{a triple}. However, this pair uniquely determines a
counterclockwise triple from the
definition, and we decided to pick a name \emph{triple} to avoid a
confusion with an inversion $\si(\square) > \si(\square')$, which is
also a pair (of entries).
\end{remark}

We define
\begin{equation}
\label{eq:inv}
\inv(\si) := \#\Inv(\si)-\sum_{\square \in \Des(\si)}a_\la(\square) = \#\InvP(\si),
\end{equation}
where the second equality was shown in \cite{HaglundHaimanLoehr2005}.

It turned out that the statistics $\maj$, and $\inv$ can be used to describe the combinatorics of the Macdonald
polynomials, by the following explicit combinatorial formula, which
from now on we
treat as the definition of Macdonald polynomials:

\begin{theorem}{\cite{HaglundHaimanLoehr2005}}
\label{theo:Haglund}
\begin{equation}
\label{eq:Haglund}
\tilde{H}_{\la}(\bm{x};q,t) = \sum_{\si:\la\to\N_+}q^{\inv(\si)}\
t^{\maj(\si)}\ \xx^\si.
\end{equation}
\end{theorem}

Finally, we say that a filling $\si: \la \to \N_+$ is \emph{standard}
if the set of entries corresponds to the set $[|\la|]$. Note that standard fillings of
$\la$ are in a natural bijection with permutations from $\Sym{n}$ given by the
correspondence $\si \leftrightarrow w_\si$, where $w_\si$ is the
reading word of $\si$.



\section{Coloring of the Young diagram $\la^{[r]}$ and graphs}
\label{sec:Coloring}

\subsection{Coloring of the Young diagram $\la^{[r]}$}
\label{subsect:Identification}

    \begin{figure}[t]
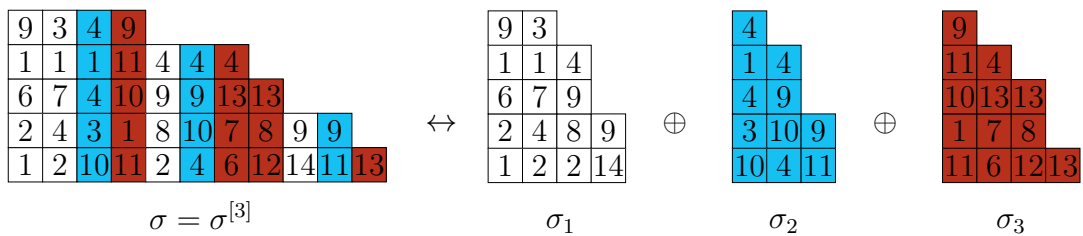

        \[
          \begin{array}{c}
        \newcommand\yb{\Yfillcolour{cyan!70}}
        \newcommand\yr{\Yfillcolour{BrickRed}}
        \newcommand\yw{\Yfillcolour{white}}
        \newcommand\eleven{11}
        \YFrench
        \young(12!\yb <10>!\yr \eleven!\yw 2!\yb 4!\yr 6<12>!\yw <14>!\yb \eleven!\yr
      <13>,!\yw 24!\yb 3!\yr 1!\yw 8!\yb <10>!\yr
      78!\yw 9!\yb 9,!\yw 67!\yb 4!\yr <10>!\yw 9!\yb 9!\yr <13><13>,!\yw
      11!\yb 1!\yr \eleven!\yw 4!\yb 4!\yr 4,!\yw 93!\yb 4!\yr 9)
        \vspace{2mm} \\
        \si = \si^{[3]}
    \end{array}
        \ \leftrightarrow\
        \begin{array}{c}
        	    \YFrench
            \young(122<14>,2489,679,114,93) \vspace{2mm} \\
        \si_1
    \end{array}\ \oplus\ \ 
    \begin{array}{c}
            \YFrench
            \Yfillcolour{cyan!70}
            \young(<10>4<11>,3<10>9,49,14,4) \vspace{2mm} \\
        \si_2
    \end{array}\ \oplus\ \ 
    \begin{array}{c}
            \YFrench
            \Yfillcolour{BrickRed}
            \young(<11>6<12><13>,178,<10><13><13>,<11>4,9) \vspace{2mm} \\
        \si_3
    \end{array}
    \vspace{-4mm}
        \]
        \caption{The diagram of an entry-wise sum of partitions.}
        \label{fig:LegArmSum}
    \end{figure}

Let $\la^1,\dots,\la^r$ be partitions, and let $\pi \in \PPP([r])$ be
a set partition. For each $B \in \pi$, we are going to color the
columns of $\la^B$ by numbers $b \in B$ as follows: we observe
that the Young diagram $\la^B$ can be constructed 
    by sorting the columns of the diagrams $\la^{b_1}$, \ldots,
    $\la^{b_t}$ in decreasing order, where $B=\{b_1,\dots,b_t\}$ and
    $b_1<\cdots<b_t$. When several columns have the same length, we
    use the total order of $B$,
    that is we put first the columns of $\la^{b_1}$,
    then those of $\la^{b_2}$ and so on. We say that a
    column of $\la^B$ is \emph{colored by $b \in B$} if this column
    is identified with the column of $\la^b$ in the above
    construction. Similarly, we say that a box $\square \in \la^B$ is
    \emph{colored by $b \in B$} if it lies in the column colored
    by $b$; see \cref{fig:LegArmSum}
    (at the moment, please disregard entries).
    This gives a way to identify boxes of $\la^{[r]}$ with the boxes of $\{\la^B:
    B \in \pi\}$. To be more precise a box $\square \in \la^B$ which
    lies in the $i$-th column colored by $b$ in $\la^B$ (not necessarily in the
    $i$-th column of $\la^B$) and in the $j$-th row of $\la^B$ is
    identified with the box $\tilde{\square}$ of $\la^{[r]}$ which
    lies in the $i$-th column colored by $b$ in $\la^{[r]}$ and in
    the $j$-th row of $\la^{[r]}$. 

This identification leads to a one-to-one correspondence between all the
fillings $\si$ of $\la^{[r]}$ with entries from a given set $\mathcal{A}$
and between the sets of fillings $\{\si^B: \la^B \to \A\ |\ B \in
\pi\}$. For a given set of fillings $\{\si_b: \la^b \to
\mathcal{A}\ |\ b \in B\}$ the corresponding filling $\si : \la^B \to
\mathcal{A}$ is denoted by $\si^B$, see \cref{fig:LegArmSum}. In particular, for any filling $\si = \si^{[r]}:
\la^{[r]} \to \N_+$, and for any set-partition $\pi \in \PPP([r])$ we have the following formula:
\begin{equation}
\label{eq:EqualityOfMaj}
\maj(\si) = \sum_{\square \in
  \Des(\si)}(\ell_{\la^{[r]}}(\square)+1)
= \sum_{B \in \pi}\sum_{\square \in
  \Des(\si^B)}(\ell_{\la^B}(\square)+1) = \sum_{B \in \pi}\maj(\si^B).
\end{equation}
Indeed, for any filling $\si : \la^B \to \N_+$ its descent set
$\Des(\si)$ decomposes as $\Des(\si) = \bigsqcup_{b \in B} \Des(\si_b)$,
and for any $\square
\in \la^B$ colored by $b$ one has
$\ell_{\la^b}'(\square) = \ell_{\la^B}'(\square) =
  \ell_{\la^{[r]}}'(\tilde{\square})$ and $\ell_{\la^b}(\square) = \ell_{\la^B}(\square) =
  \ell_{\la^{[r]}}(\tilde{\square})$.

The statistic $\inv$ is not
additive with respect to the operation $\oplus$ but its behaviour is
also very simple. Let $\InvP_1(\sigma)$ denote the set of triples
$(\square_1,\square_2) \in \InvP(\si)$ such that $\square_1,\square_2$
have the same color, and $\InvP_2(\sigma)$ denotes the set of triples
$(\square_1,\square_2) \in \InvP(\si)$ such that $\square_1,\square_2$
have different colors. Then, the set $\InvP(\sigma^B)$ of inversion triples
of the colored filling $\si^B$ decomposes as the disjoint sum
$\InvP(\sigma^B) = \InvP_1(\sigma^B)\sqcup \InvP_2(\sigma^B)$ and
\begin{equation}
\label{eq:InvIdentity}
\InvP_1(\sigma^B) = \bigsqcup_{i \in B}\InvP_1(\sigma^{\{i\}}),\quad
\InvP_2(\sigma^B) = \bigsqcup_{\{i,j\} \subset B}\InvP_2(\sigma^{\{i,j\}}).
\end{equation}

\begin{figure}[t]
\subfloat[]{
\label{subfig:A2}
    \YFrench
    \Yboxdim{1cm}
    \begin{tikzpicture}[scale=0.6, white/.style={circle,draw=black,fill=white,inner sep=1.5pt},
	blue/.style={circle,draw=black,fill=cyan!70,inner sep=1.5pt}, red/.style={circle,draw=black,fill=BrickRed,inner sep=1.5pt}]
        \newcommand\yb{\Yfillcolour{cyan!70}}
        \newcommand\yr{\Yfillcolour{BrickRed}}
        \newcommand\yw{\Yfillcolour{white}}
            
        \tgyoung(0cm,0cm,;;;;;;;;;;;,;;;;;;;;;;,;;;;;;;;,;;;;;;;,;;;;)

        \draw[yellow, thick] (0.5,4.5) to (1.5,4.5);
        \draw[gray, thick] (0.5,4.5) to [bend left] (2.5,4.5);
        \draw[gray, thick] (0.5,3.5) to [bend left] (3.5,3.5);
        \draw[gray, thick] (1.5,3.5) to [bend left] (3.5,3.5);
        \draw[gray, thick] (2.5,3.5) to (3.5,3.5);
        \draw[gray, thick] (0.5,3.5) to [bend left] (3.5,3.5);
        \draw[gray, thick] (0.5,2.5) to [bend left] (2.5,2.5);
        \draw[gray, thick] (3.5,2.5) to (4.5,2.5);
        \draw[gray, thick] (3.5,2.5) to [bend left] (5.5,2.5);
        \draw[gray, thick] (1.5,1.5) to (2.5,1.5);
        \draw[gray, thick] (2.5,1.5) to (3.5,1.5);
        \draw[gray, thick] (4.5,1.5) to [bend left]  (6.5,1.5);
        \draw[gray, thick] (5.5,1.5) to (6.5,1.5);
        \draw[gray, thick] (5.5,1.5) to [bend left] (7.5,1.5);
        \draw[gray, thick] (5.5,1.5) to [bend left] (8.5,1.5);
        \draw[yellow, thick] (5.5,1.5) to [bend left] (9.5,1.5);
        \draw[gray, thick] (2.5,0.5) to [bend left] (4.5,0.5);
        \draw[yellow, thick] (2.5,0.5) to [bend left] (5.5,0.5);
        \draw[gray, thick] (2.5,0.5) to [bend left] (6.5,0.5);
        \draw[gray, thick] (3.5,0.5) to (4.5,0.5);
        \draw[gray, thick] (3.5,0.5) to [bend left] (5.5,0.5);
        \draw[yellow, thick] (3.5,0.5) to [bend left] (6.5,0.5);
        \draw[gray, thick] (7.5,0.5) to [bend left] (9.5,0.5);
        \draw[gray, thick] (8.5,0.5) to (9.5,0.5);
        \draw[gray, thick] (8.5,0.5) to [bend left] (10.5,0.5);

		\draw (0.5,2.5) node [white] {};
		\draw (0.5,3.5) node [white] {};
		\draw (0.5,4.5) node [white] {};
		\draw (1.5,1.5) node [white] {};
		\draw (1.5,3.5) node [white] {};
		\draw (1.5,4.5) node [white] {};
		\draw (4.5,0.5) node [white] {};
		\draw (4.5,1.5) node [white] {};
		\draw (4.5,2.5) node [white] {};
		\draw (8.5,0.5) node [white] {};
		\draw (8.5,1.5) node [white] {};
		\draw (2.5,0.5) node [blue] {};
		\draw (2.5,1.5) node [blue] {};
		\draw (2.5,2.5) node [blue] {};
		\draw (2.5,3.5) node [blue] {};
		\draw (2.5,4.5) node [blue] {};
		\draw (5.5,0.5) node [blue] {};
		\draw (5.5,1.5) node [blue] {};
		\draw (5.5,2.5) node [blue] {};
		\draw (9.5,0.5) node [blue] {};
		\draw (9.5,1.5) node [blue] {};
		\draw (3.5,0.5) node [red] {};
		\draw (3.5,1.5) node [red] {};
		\draw (3.5,2.5) node [red] {};
		\draw (3.5,3.5) node [red] {};
		\draw (6.5,0.5) node [red] {};
		\draw (6.5,1.5) node [red] {};
		\draw (7.5,0.5) node [red] {};
		\draw (7.5,1.5) node [red] {};
		\draw (10.5,0.5) node [red] {};
\end{tikzpicture}}
\hspace{-4pt}
\subfloat[]{
\label{subfig:A3}
    \begin{tikzpicture}[scale=0.3, white/.style={circle,draw=black,fill=white,inner sep=3pt},
	blue/.style={circle,draw=black,fill=cyan!70,inner sep=3pt}, red/.style={circle,draw=black,fill=BrickRed,inner sep=3pt}]

		\draw (3,0) node (W) [white] {};
		\draw (0,5) node (B) [blue] {};
		\draw (6,5) node (R) [red] {};

                \draw[gray, thick] (W) to node[anchor=east] {$6$} (B);
                \draw[gray, thick] (W) to node[anchor=west] {$6$} (R);
                \draw[gray, thick] (B) to node[anchor=south] {$8$} (R);
                \draw[yellow, thick] (B) to 
                [out=120,in=160,looseness=12] (B);
                \draw[yellow, thick] (B) to
                [out=100,in=180,looseness=15] (B);
                \draw[yellow, thick] (R) to
                [out=0,in=80,looseness=15] (R);
                \draw[yellow, thick] (W) to [out=310,in=230,looseness=15] (W);

\end{tikzpicture}}
        \caption{The gray and yellow edges in \cref{subfig:A2} represent
          inversion triples in $\InvP_2(\si)$ and in $\InvP_1(\si)$,
          respectively, and colors of their 
          endpoints for $\si$ from
          \cref{fig:LegArmSum}. \cref{subfig:A3} presents a graph
          $G^\si_{\la^1,\la^2,\la^3}$ obtained from the edges from \cref{subfig:A2} by
          identifying the vertices of the same color. The labels on the edges of
          $G^\si_{\la^1,\la^2,\la^3}$ indicate their multiplicities.}
        \label{fig:Multigraph1}
    \end{figure}

Let $\sigma: \la^{[r]} \to \N_+$ be a filling. We are ready to
construct the
graph $G^\si_{\la^1,\dots,\la^r} := (V,E)$.  For each inversion triple in $\sigma$, we draw an edge
linking its boxes, and
we color its endpoints by the colors of these boxes
from the colored diagram $\la^{[r]}$; then we identify all the endpoints
of the same color -- see \cref{fig:Multigraph1} for a construction of
$G^\si_{\la^1,\dots,\la^r}$ for $r=3$ and $\si,\la^1,\la^2,\la^3$ as
in \cref{fig:LegArmSum}.
More formally, $G^\si_{\la^1,\dots,\la^r} := (V,E)$ is defined by the
following data:
\begin{enumerate}[label=(G\arabic*)]
\item
\label{eq:graph1a}
the set of vertices $V$ is equal to $[r]$;
\item
\label{eq:graph1b}
$e_{i,j}(G^\si_{\la^1,\dots,\la^r}) = \begin{cases}
  \#\InvP_1(\si^{\{i\}}) &\text{ for $i=j$},\\ \#\InvP_2(\si^{\{i,j\}})
  &\text{ for $i\neq j$.}\end{cases}$
\end{enumerate}

\subsection{Proof of  \cref{theo:MacdonaldCumuFormula}}

We recall the formula \eqref{eq:MacdonaldCumuFormula} that we need to prove: 
\[ \ka(\la^1,\dots,\la^r) = \sum_{\si: \la^{[r]} \to
  \N_+}\I_{G^\si_{\la^1,\dots,\la^r}}(q)\ t^{\maj(\sigma)}\ \xx^\sigma.\]

\begin{proof}[Proof of  \cref{theo:MacdonaldCumuFormula}]
Using definition of Macdonald
cumulants (\eqref{eq:DefCumulants} and \eqref{eq:MacdoCumu}) and HHL's
formula \eqref{eq:Haglund}, we rewrite the left hand side of \eqref{eq:MacdonaldCumuFormula}
as follows:
\begin{multline*} 
(q-1)^{1-r}\sum_{\substack{\pi \in \PPP([r]) } } 
            (-1)^{\#\pi-1}(\#\pi-1)! \prod_{B \in \pi}
            \tilde{H}_{\la^B}(\bm{x};q,t) \\
= (q-1)^{1-r}\sum_{\substack{\pi \in \PPP([r]) } } 
            (-1)^{\#\pi-1}(\#\pi-1)! \sum_{\si:\la^{[r]}\to \N_+\ }\prod_{B \in \pi}
            q^{\inv(\si^B)}\ t^{\maj(\si^B)}\ \xx^{\si_B}\\
= \sum_{\si:\la^{[r]}\to \N_+\ }t^{\maj(\si)}\left((q-1)^{1-r}\sum_{\substack{\pi \in \PPP([r]) } } 
            (-1)^{\#\pi-1}(\#\pi-1)! \prod_{B \in \pi}
            q^{\inv(\si^B)}\right)\xx^{\si}.
\end{multline*}
The first equality is a consequence of the one-to-one correspondence between
fillings of a given diagram and the sets of fillings of its
subdiagrams described in \cref{subsect:Identification}, while the last equality follows from
\eqref{eq:EqualityOfMaj}. The expression in parentheses is given by the following formula:
\[ (q-1)^{1-\#V} \sum_{\substack{\pi \in \PPP(V) } } 
            (-1)^{\#\pi-1}(\#\pi-1)! \prod_{B \in \pi} q^{\#E|_B},\]
where $(V,E)  = G^\si_{\la^1,\dots,\la^r}$,
which is equal to $\I_{G^\si_{\la^1,\dots,\la^r}}(q)$ by \cref{lem:TutteCumulant} and \eqref{eq:wow}. Indeed, strictly from the
definition \eqref{eq:graph1a}--\eqref{eq:graph1b} of $G^\si_{\la^1,\dots,\la^r} $ one has $V = [r]$, and
\[ \#E|_B = \sum_{i \in B}\#\InvP_1(\sigma^{\{i\}}) +
  \sum_{\{i,j\} \subset B}\#\InvP_2(\sigma^{\{i,j\}}) = \#\InvP_1(\sigma^B) + \#\InvP_2(\sigma^B) = \inv(\si^B),\]
where the second equality is given by
\eqref{eq:InvIdentity}. This concludes the proof.
\end{proof}

\section{Fundamental quasisymmetric function expansion}
\label{sec:Fundamental}

In this section we are going to find a formula for Macdonald cumulants
in terms of fundamental quasisymmetric functions and their
superization by applying the method from \cite[Section 4]{HaglundHaimanLoehr2005}.

\subsection{Fundamental quasisymmetric functions}

\begin{definition}[\cite{Gessel1984}] 
For any nonnegative integer $n$ and a subset $D \subset [n-1]$ a \emph{fundamental quasisymmetric
function} $F_{n,D}(\xx)$ of degree $n$ in variables $\xx = x_1,
x_2,\dots$ is defined by the formula
\[ F_{n,D}(\xx) := \sum_{\substack{i_1\leq \cdots \leq i_n\\j \in D \Rightarrow
    i_j < i_{j+1}}}x_{i_1}\cdots x_{i_n}.\]
\end{definition}

More generally, let $\A = \Z_+\cup\Z_- = \{1,\bar{1},2,\bar{2},\dots\}$
be a ``super'' alphabet of \emph{positive letters $i$} and \emph{negative letters
  $\bar{i}$}, and let $(\A,\leq)$ be a total order of $\A$ which
preserves the natural order of positive integers. The ``super''
quasisymmetric function $\tilde{F}_{n,D}(\xx,\yy)$ in variables $\xx =
x_1,x_2,\dots$ and $\yy = y_1,y_2,\dots$ is defined by
\[ \tilde{F}_{n,D}(\xx,\yy) := \sum_{\substack{i_1\leq \cdots \leq i_n\\
    i_j = i_{j+1} \in \Z_+ \Rightarrow j \notin D\\ i_j = i_{j+1} \in \Z_- \Rightarrow j \in D
}}z_{i_1}\cdots z_{i_n},\]
where the indices $i_1,\dots,i_n$ run over $\A$, and we set $z_i =
x_i$ for $i \in \Z_+$, and $z_{\bar{i}} = y_i$ for $\bar{i} \in \Z_-$.

\begin{definition}[\cite {HaglundHaimanLoehrRemmelUlyanov2005}] 
The \emph{superization} of a symmetric function $f(\xx)$ is
\[\tilde{f}(\xx, \yy) := \omega_{\YY} f[\XX + \YY]\]
(the subscript $\YY$ denotes that $\omega$ acts on $f[\XX + \YY] = f(\xx, \yy)$ considered as a symmetric
function of the $\yy$ variables only).
\end{definition}


\begin{proposition}[\cite{HaglundHaimanLoehrRemmelUlyanov2005}]
\label{prop:HHLRU}
Let $f(\xx)$ be a homogeneous symmetric function of degree $n$, written in
terms of fundamental quasisymmetric functions as
\[f(\xx) = \sum_{D}c_D\ F_{n,D}(\xx).\]
Then its ''superization'' is given by
\[ \tilde{f}(\xx, \yy) = \sum_{D}c_D\ \tilde{F}_{n,D}(\xx,\yy).\]
\end{proposition}


\subsection{Fundamental quasisymmetric function expansion and super fillings}

For any pair of letters $x,y \in (\A,\leq)$ and for any sign $\bullet \in \{+,-\}$
we write $x \leq_\bullet y$ when $x < y$, or $x = y \in
\Z_\bullet$. We define $\geq_\bullet$ similarly.

Given a super alphabet $\A$, a \emph{super filling} of $\mu$ is a function
$\si: \mu \to \A$. 

We define the set $\Des(\si)$ of boxes $\square \in \mu$
occurring as the upper box in a \emph{descent}, that is in the pair
$\si(\square) \geq_- \si(\square')$, where $\square$ lies directly above $\square'$ in $\mu$.
The entries of an attacking pair $(\square,\square')$ such that $\si(\square) \geq_-
\si(\square')$ and such that $\square$ precedes $\square'$
in the reading order form an \emph{inversion}. The set of positions of
all inversions in $\si$ is denoted by $\Inv(\si)$, as before. 

We say that the ordered triple of boxes
$\square_1,\square_2,\square_3$ is \emph{counterclockwise
  increasing}, if one of the following conditions holds
true:
\begin{itemize}
\item
$\si(\square_1) \leq_+ \si(\square_3) \leq_- \si(\square_2)$, or
\item
$\si(\square_3) \leq_- \si(\square_2) \leq_- \si(\square_1)$, or
\item
$\si(\square_2) \leq_- \si(\square_1) \leq_+ \si(\square_3)$.
\end{itemize} 
We define the \emph{inversion triple} as a pair of boxes $(\square_1,\square_2)$, where $\square_1$ is a box lying in the same
row as $\square_2$ to its left and such that a triple $\square_1,\square_2,\square_3$ is  \emph{counterclockwise
  increasing}, where $\square_3$ is a box lying directly below $\square_1$.

The statistics $\inv(\si)$ and $\maj(\si)$ are defined in terms of
$\Inv(\si)$, $\Des(\si)$ and $\InvP(\si)$ by \eqref{eq:major}, and
\eqref{eq:inv} as for ordinary fillings. 
Given a permutation $\si \in \Sym{n}$ and an integer $i < n$, we say that $i$ is an \emph{inverse descent} of $\si$ if
$i+1$ lies to the left of $i$ in $\si$. Let $\iDes(\si)$ denote the
set of inverse descents of $\si$.

\vspace{10pt}

Let $\la^1,\dots,\la^r$ be partitions and let $n= |\la^1|+\cdots+|\la^r|$. Then, using formula
\eqref{eq:MacdonaldCumuFormula} and a verbatim argumentation as in
\cite[Section 4]{HaglundHaimanLoehr2005} we have the expansion
\begin{equation}
\label{eq:CumulantInQuasisymmetric}
\ka(\la^1,\dots,\la^r)(\xx) = \sum_{\si \in \Sym{n}}\I_{G^{\si}_{\la^1,\dots,\la^r}}(q)\
t^{\maj_{\la^{[r]}}(\si)}\ F_{n,\iDes(\si)}(\xx),
\end{equation}
where we abuse notation by denoting both a permutation by $\si$, and
the associated standard filling of
$\la^{[r]}$ with the reading word given by $\si$ (see \cref{subsec:Fillings}). Thus, by \cref{prop:HHLRU}
\[ \tilde{\ka}(\la^1,\dots,\la^r)(\xx,\yy) = \sum_{\si \in \Sym{n}}\I_{G^{\si}_{\la^1,\dots,\la^r}}(q)\
t^{\maj_{\la^{[r]}}(\si)}\ \tilde{F}_{n,\iDes(\si)}(\xx,\yy),\]
or equivalently (which will be more useful in applications)
\begin{equation}
\label{eq:CumulantInSuperQuasisymmetric}
\tilde{\ka}(\la^1,\dots,\la^r)(\xx,-\yy) =
\tilde{\ka}(\la^1,\dots,\la^r)[\XX-\YY] = \sum_{\si:\la^{[r]} \to \A}\I_{G^{\si}_{\la^1,\dots,\la^r}}(q)\
t^{\maj_{\la^{[r]}}(\si)}\ \zz^T,
\end{equation}
where we sum over all super fillings of $\la^{[r]}$, and  $z_i =
x_i$ for $i \in \Z_+$, and $z_{\bar{i}} = - y_i$ for $\bar{i} \in
\Z_-$. The first equality in \eqref{eq:CumulantInSuperQuasisymmetric}
follows from \eqref{eq:pleth}.

\section{Multivariate $q,t$-Kostka coefficients for hooks}
\label{sec:Kostka}

    \begin{figure}[t]
\[
          \begin{array}{c}
        \YFrench
    \Yboxdim{1cm}
    \begin{tikzpicture}[scale=0.5]
        \newcommand\yb{\Yfillcolour{cyan!70}}
        \newcommand\yr{\Yfillcolour{BrickRed}}
        \newcommand\yw{\Yfillcolour{white}}
        \newcommand\eleven{11}
        \tgyoung(0cm,0cm,;;!\yb ;!\yr ;!\yw ;!\yb ;!\yr ;;!\yw ;!\yb ;!\yr
      ;,!\yw ;;!\yb ;!\yr ;!\yw ;!\yb ;!\yr
      ;;!\yw ;!\yb ;,!\yw ;;!\yb ;!\yr ;!\yw ;!\yb ;!\yr ;;,!\yw
      ;;!\yb ;!\yr ;!\yw ;!\yb ;!\yr ;,!\yw ;;!\yb ;!\yr ;)
       \Ylinethick{1.5pt}
       \tgyoung(1cm,1cm,!\yw ;)
       \tgyoung(5cm,2cm,!\yb ;)
       \tgyoung(3cm,3cm,!\yr ;)
\end{tikzpicture}
        \vspace{2mm} \\
        \la^{[3]}
    \end{array}
\
        \begin{array}{c}
        \YFrench
    \Yboxdim{1cm}
    \begin{tikzpicture}[scale=0.5, white/.style={circle,draw=black,fill=white,inner sep=1.5pt},
	blue/.style={circle,draw=black,fill=cyan!70,inner sep=1.5pt}, red/.style={circle,draw=black,fill=BrickRed,inner sep=1.5pt}]
        \newcommand\yb{\Yfillcolour{cyan!70}}
        \newcommand\yr{\Yfillcolour{BrickRed}}
        \newcommand\yw{\Yfillcolour{white}}
            
        \tgyoung(0cm,0cm,;;;;;;;;;;;,;;;;;;;;;;,;;;;;;;;,;;;;;;;,;;;;)

        \draw[gray, thick] (0.5,3.5) to [bend left] (3.5,3.5);
        \draw[gray, thick] (1.5,3.5) to [bend left] (3.5,3.5);
        \draw[gray, thick] (2.5,3.5) to (3.5,3.5);
        \draw[gray, thick] (0.5,3.5) to [bend left] (3.5,3.5);
        \draw[gray, thick] (0.5,2.5) to [bend left] (5.5,2.5);
        \draw[gray, thick] (1.5,2.5) to [bend left] (5.5,2.5);
        \draw[gray, thick] (2.5,2.5) to [bend left] (5.5,2.5);
        \draw[gray, thick] (3.5,2.5) to [bend left] (5.5,2.5);
        \draw[gray, thick] (4.5,2.5) to (5.5,2.5);
        \draw[gray, thick] (0.5,1.5) to (1.5,1.5);

		\draw (0.5,2.5) node [white] {};
		\draw (1.5,2.5) node [white] {};
		\draw (0.5,3.5) node [white] {};
		\draw (1.5,1.5) node [white] {};
		\draw (0.5,1.5) node [white] {};
		\draw (1.5,3.5) node [white] {};
		\draw (4.5,2.5) node [white] {};
		\draw (2.5,2.5) node [blue] {};
		\draw (2.5,3.5) node [blue] {};
		\draw (5.5,2.5) node [blue] {};
		\draw (3.5,2.5) node [red] {};
		\draw (3.5,3.5) node [red] {};
\end{tikzpicture} \vspace{2mm} \\
\
\end{array}
        \begin{array}{c}
    \begin{tikzpicture}[scale=0.35, white/.style={circle,draw=black,fill=white,inner sep=3pt},
	blue/.style={circle,draw=black,fill=cyan!70,inner sep=3pt}, red/.style={circle,draw=black,fill=BrickRed,inner sep=3pt}]

		\draw (3,1) node (W) [white] {};
		\draw (0,6) node (B) [blue] {};
		\draw (6,6) node (R) [red] {};

                \draw[gray, thick] (W) to (B);
                \draw[gray, thick] (W) to [bend left=16] (B);
                \draw[gray, thick] (W) to [bend right=16] (B);
                \draw[gray, thick] (W) to [bend left=8] (R);
                \draw[gray, thick] (W) to [bend right=8] (R);
                \draw[gray, thick] (B) to [bend left=8] (R);
                \draw[gray, thick] (B) to [bend right=8] (R);
                \draw[gray, thick] (B) to
                [out=100,in=180,looseness=15] (B);
                \draw[gray, thick] (W) to
                [out=310,in=230,looseness=15] (W);

\end{tikzpicture} \\
        \ \ \ \ \G_{\la^1,\la^2,\la^3}^{\square_1,\square_2,\square_3}
    \end{array}
    \vspace{-2mm}
        \]
        \caption{The graph
          $G_{\la^1,\la^2,\la^3}^{\square_1,\square_2,\square_3}$ for
          $\la^1,\la^2,\la^3$ from \cref{fig:LegArmSum}.}
        \label{fig:graph}
    \end{figure}
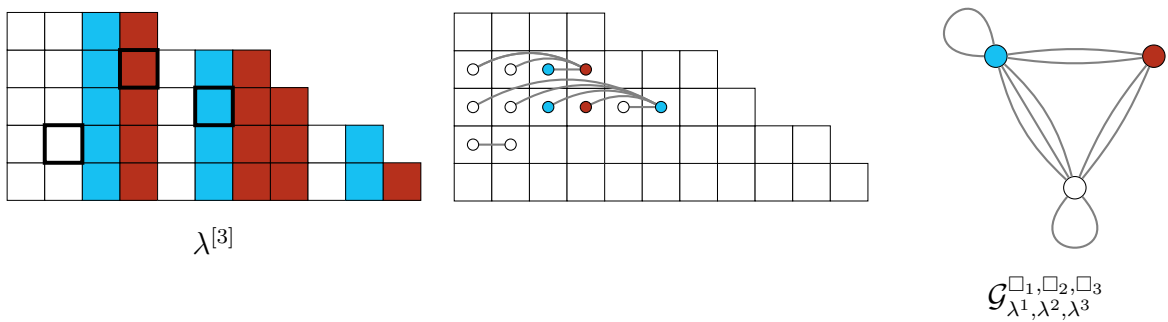

Let $\la^1,\dots,\la^r$ be partitions, and let $1 \leq s \leq
|\la^1|+\cdots+|\la^r|$ be a positive integer. For any
subset $\{\square_1,\dots,\square_s\} \subset \la^{[r]}$ of boxes we
construct a graph $G^{\square_1,\dots,\square_s}_{\la^1,\dots,\la^r} :=
(V,E)$ as follows: we draw an edge between each $\square_i$ and each
box to its left lying in the same row, and
we color its endpoints by the colors of the corresponding boxes
in $\la^{[r]}$; then we identify all the endpoints of the same
color -- see \cref{fig:graph} for a construction of
$G^{\square_1,\square_2,\square_3}_{\la^1,\dots,\la^r}$ for $r=3$ and $\la^1,\la^2,\la^3$ as in
\cref{fig:LegArmSum}. In other words
\begin{itemize}
\item
the set of vertices $V$ is equal to $[r]$, 
\item
the number of edges linking vertices $i,j \in V$ is equal to 
the number of pairs $(\square_k,\square')$ such that $\square'$ is in the same row as $\square_k$ to its left, and the pair
$(\square_k,\square')$ is colored by $\{i,j\}$, where $1 \leq k \leq s$.
\end{itemize}

We are ready to prove \cref{conj:GeneralizedKostka} in the case of hooks.

\begin{theorem}
\label{theo:HookKostka}
Let $\la^1,\ldots,\la^r$ be partitions with
$|\la^{[r]}|=n$. Then, for any nonnegative integer $s$,
the coefficient of $(-u)^s$ in $\ka(\la^1,\dots,\la^r)[1-u]$ is equal
to
\begin{equation}
\label{eq:claim} 
\ka(\la^1,\dots,\la^r)[1-u]|_{(-u)^s} = \sum_{\{\square_1,\dots,\square_s\} \subset
  \la^{[r]}}\I_{G_{\la^1,\dots,\la^r}^{\square_1,\dots,\square_s}}(q)
\ t^{\sum_{1 \leq i
    \leq s}
    \ell'_{\la^{[r]}}(\square_i)}.
\end{equation}
Equivalently, the multivariate $q,t$-Kostka number
$\widetilde{K}_{(n-s,1^s); \la^1,\dots,\la^r}(q,t)$ is a polynomial in $q,t$ with
nonnegative integer coefficients given by the following formula:
\begin{equation}
\label{eq:HookKostka}
\widetilde{K}_{(n-s,1^s); \la^1,\dots,\la^r}(q,t)
=\sum_{\{\square_1,\dots,\square_s\} \subset
  \la^{[r]}\setminus(1,1)}\I_{G_{\la^1,\dots,\la^r}^{\square_1,\dots,\square_s}}(q)\
t^{\sum_{1 \leq i
    \leq s}
    \ell'_{\la^{[r]}}(\square_i)}.
\end{equation}
\end{theorem}

\begin{proof}
Firstly, we sketch an argument of Macdonald showing the equivalence of
formulas \eqref{eq:claim} and \eqref{eq:HookKostka}.
Macdonald proved \cite[Section VI.8,
Example 2]{Macdonald1995} that
\[s_\la[1-u] = \begin{cases} 0 &\text{ if $\lambda$ is not a hook},\\
    (1-u)(-u)^s &\text{ if $\lambda = (n-s,1^s)$}.\end{cases}\]
Thus, for any homogeneous symmetric function
\[ f := \sum_{\lambda \vdash n} c_\la\ s_\la,\]
we obtain the following relation
\begin{equation} 
\label{eq:relation}
f[1-u]|_{(-u)^s} = c_{(n-s,1^s)} + c_{(n-s+1,1^{s-1})}.
\end{equation}
Note that for any subset $\{\square_1,\dots,\square_s\} \subset
  \la^{[r]}\setminus(1,1)$ graphs
  $G_{\la^1,\dots,\la^r}^{\square_1,\dots,\square_s}$ and $G_{\la^1,\dots,\la^r}^{\square_1,\dots,\square_s,(1,1)}$
coincide, which is clear from our construction (there are no boxes to the left of $(1,1)$). Moreover, $\ell'_{\la^{[r]}}((1,1)) =
0$, therefore
\begin{multline*}
\sum_{\{\square_1,\dots,\square_s\} \subset
  \la^{[r]}}\I_{G_{\la^1,\dots,\la^r}^{\square_1,\dots,\square_s}}(q)
\ t^{\sum_{1 \leq i
    \leq s}
    \ell'_{\la^{[r]}}(\square_i)} \\
= \sum_{\{\square_1,\dots,\square_s\} \subset
  \la^{[r]}\setminus(1,1)}\I_{G_{\la^1,\dots,\la^r}^{\square_1,\dots,\square_s}}(q)
\ t^{\sum_{1 \leq i
    \leq s}
    \ell'_{\la^{[r]}}(\square_i)} \\
+ \sum_{\{\square_1,\dots,\square_{s-1}\} \subset
  \la^{[r]}\setminus(1,1)}\I_{G_{\la^1,\dots,\la^r}^{\square_1,\dots,\square_{s-1}}}(q)
\ t^{\sum_{1 \leq i
    \leq s-1}
    \ell'_{\la^{[r]}}(\square_i)},
\end{multline*}
and we proved that the relation \eqref{eq:relation} is satisfied by
\eqref{eq:claim} and \eqref{eq:HookKostka}. This proves that formulas
\eqref{eq:claim} and \eqref{eq:HookKostka} are equivalent.

In order to compute the coefficient of $(-u)^s$ in
$\ka(\la^1,\dots,\la^r)[1-u]$ we use formula
\eqref{eq:CumulantInSuperQuasisymmetric}, which says that this coefficient is equal
to the sum of $\I_{G^{\si}_{\la^1,\dots,\la^r}}(q)\
t^{\maj_{\la^{[r]}}(\si)}$ over
super fillings $\si$ with $n-s$ entries equal to $1$ and $s$ entries
equal to $\bar{1}$. If we use an ordering of $\A$ in which
$\bar{1} < 1$ then the set of inversion triples
consists of the pairs $(\square_1,\square_2)$ such that
$\si(\square_2) = \bar{1}$ and $\si(\square_1)$ is arbitrary (and, as
a part of the definition, $\square_1$ is in the same row as
$\square_2$ to its left). Moreover, $\square \in
\Des(\si)$ if and only if $\si(\square') =
\bar{1}$ for $\square'$ strictly below $\square$. Let $\si$ be a
super filling of $\la^{[r]}$ with $s$
boxes
$\square_1,\dots,\square_s$ filled by $\bar{1}$ and other boxes filled
by $1$. It is obvious from the construction that $G^{\si}_{\la^1,\dots,\la^r} =
  G_{\la^1,\dots,\la^r}^{\square_1,\dots,\square_s}$.
Indeed, we recall that $G^{\si}_{\la^1,\dots,\la^r}$ is constructed by
replacing inversion
triples by edges and identifying vertices of the same color. Inversion triples in $\si$ are given by $(\square,\square_i)$,
where $1 \leq i \leq s$, and $\square$ is an arbitrary box in the same row as $\square_i$ to its left. Thus, the construction of both
$G^{\si}_{\la^1,\dots,\la^r}$ and
$G_{\la^1,\dots,\la^r}^{\square_1,\dots,\square_s}$ coincides.
Let $\si$ be a super filling of $\la^{[r]}$ with $s$
boxes
$\square_1,\dots,\square_s$ filled by $\bar{1}$ and other boxes filled
by $1$. We associate with it another super filling $\widehat{\si}$ of $\la^{[r]}$ with $s$
boxes
$\widehat{\square_1},\dots,\widehat{\square_s}$ filled by $\bar{1}$ and other boxes filled
by $1$, by setting $\widehat{(x,y)} := \left(x,(\la^{[r]})^t_x+1-y\right) \in
\la^{[r]}$. In other words $\widehat{(x,y)}$ is a box lying in the same column
as $(x,y)$, but in the
$y$-th row counting from the top of this
column. The operation $\widehat{}$ is an involution on the set of
super fillings of $\la^{[r]}$ with $s$ entries equal to $\bar{1}$ and $n-s$
entries equal to $1$. Note that $\ell_{\la^{[r]}}(\square) =
\ell'_{\la^{[r]}}(\widehat{\square})$, and $\maj_{\la^{[r]}}(\si) = \sum_{1 \leq i
    \leq s}
    \ell_{\la^{[r]}}(\square_i)$, so
\[ \maj_{\la^{[r]}}(\widehat{\si}) = \sum_{1 \leq i
    \leq s}
    \ell_{\la^{[r]}}(\widehat{\square_i}) = \sum_{1 \leq i
    \leq s}
    \ell'_{\la^{[r]}}(\square_i).\]
Finally, it is straightforward from the construction of
$G_{\la^1,\dots,\la^r}^{\square_1,\dots,\square_s}$ that for a fixed integer
$1 \leq i \leq s$ one can replace the box $\square_i$ by any other
box $\square'$ from the same column of $\la^{[r]}$ and the resulting
graph
$G_{\la^1,\dots,\la^r}^{\square_1,\dots,\square',\dots,\square_s}$ is
the same as the initial one
$G_{\la^1,\dots,\la^r}^{\square_1,\dots,\square_s}$. Indeed, for any
color $j \in [r]$ the number of boxes colored by $j$ and lying in
the same row as $\square_i$ to its left is the same as the number of boxes colored by $j$ and lying in
the same row as $\square'$ to its left. In particular
\[ G_{\la^1,\dots,\la^r}^{\square_1,\dots,\square_s} = G_{\la^1,\dots,\la^r}^{\widehat{\square_1},\dots,\widehat{\square_s}}.\]
Concluding, we can compute the coefficient in question as follows
\begin{multline*}
\sum_{\si: \la^{[r]} \to \{1^{n-s},\bar{1}^s\}}\I_{G^{\si}_{\la^1,\dots,\la^r}}(q)\
t^{\maj_{\la^{[r]}}(\si)} = \sum_{\si: \la^{[r]} \to \{1^{n-s},\bar{1}^s\}}\I_{G^{\si}_{\la^1,\dots,\la^r}}(q)\
t^{\maj_{\la^{[r]}}(\widehat{\si})} \\
= 
\sum_{\{\square_1,\dots,\square_s\} \subset
  \la^{[r]}}\I_{G_{\la^1,\dots,\la^r}^{\square_1,\dots,\square_s}}(q)
\ t^{\sum_{1 \leq i
    \leq s}
    \ell'_{\la^{[r]}}(\square_i)},
\end{multline*}
which proves \eqref{eq:claim} and finishes the proof of \cref{theo:HookKostka}.
\end{proof}

\section{Fully colored Macdonald polynomials}
\label{sec:FullyColored}

In this section we focus on the special case of cumulants
$\ka(\la^1,\dots,\la^r)$ where all the partitions $\la^1,\dots,\la^r$ are columns. These
cumulants are directly related to the cumulants we used in our
previous work \cite{DolegaFeray2016}, where we proved polynomiality
part of the $b$-conjecture, and we believe that studying their
structure might be an important step toward resolving the
$b$-conjecture. Moreover, they seem to carry many remarkable
properties and therefore they might be of an independent interest.

\begin{definition}
\label{def:ColoredMacdonald}
For any partition $\mu$, we define \emph{fully colored Macdonald
polynomial} $\bar{H}_\mu(\xx;q,t)$ as follows:
\[\bar{H}_\mu(\xx;q,t) := \ka(1^{\mu^t_1}, 1^{\mu^t_2},\dots, 1^{\mu^t_{\mu_1}}).\]
\end{definition}

\begin{theorem}
\label{theo:basis}
The family of fully colored Macdonald
polynomials $\bar{H}_\mu(\xx;q,t)$ is a linear basis of the algebra $\Lambda$
of
symmetric functions.
\end{theorem}

In order to prove above theorem we find an explicit combinatorial
formula for some plethystic substitution in the cumulants of Macdonald
polynomials. 

Let $\si :\mu \to \A$ be a super filling of $\mu$. We say that it is
\emph{compatible with $\mu$} if $|\si(x,y)| \geq y$ for all $(x,y) \in
\mu$. We also denote by $m(\si)$ and $p(\si)$, respectively, the number
of negative and positive, respectively, entries in $\mu$. We fix the following ordering of $\A$:
\[ 1 < 2 < \cdots < \bar{2} < \bar{1},\] 
and we
set $\xx^{|\si|} := \prod_{\square \in \mu}x_{|\si(\square)|}$.

\begin{lemma}
For any positive integer $r$ and partitions $\la^1,\dots,\la^r$ we
have
\begin{equation}
\label{eq:cumulantPlethT}
\ka(\la^1,\dots,\la^r)[\XX(t-1);q,t] = \sum_{\substack{\si: \la^{[r]} \to
    \A\\\text{compatible with
    }\la^{[r]}}}\I_{G^\si_{\la^1,\dots,\la^r}}(q)\ (-1)^{m(\si)}\
t^{p(\si)+\maj(\si)}\ \xx^{|\si|}.
\end{equation}
\end{lemma}

\begin{proof}
The proof of \cite[Lemma 5.2]{HaglundHaimanLoehr2005} which
corresponds to the case $r=1$ works without any changes in the
general case, so we only recall the main argument. Using
\eqref{eq:CumulantInSuperQuasisymmetric} we get the formula
\[ \ka(\la^1,\dots,\la^r)[\XX(t-1);q,t] = \sum_{\si: \la^{[r]} \to
    \A}\I_{G^\si_{\la^1,\dots,\la^r}}(q)\ (-1)^{m(\si)}\
t^{p(\si)+\maj(\si)}\ \xx^{|\si|}.\]
Therefore, it is enough to show that there exists an involution $\phi$
on the set of super fillings $\si: \la^{[r]} \to
    \A$ which fixes super fillings compatible with $\mu$, and for other
    super fillings preserves $G^\si_{\la^1,\dots,\la^r}$ and
    $p(\si)+\maj(\si)$ but increases/decreases $m(\si)$ by one.
Let $\si$ be a super filling not compatible with $\la^{[r]}$, and let
$a$ be the smallest integer such that $a = |\si(x,y)| < y$ for some
$(x,y) \in \la^{[r]}$. Let $\square' \in \la^{[r]}$ be the first box
in the reading order with $|\si(\square')| = a$. We define
\[ \phi(\si)(\square) = \begin{cases}
\si(\square) &\text{ for } \square \neq \square',\\
\overline{\si(\square)} &\text{ for } \square = \square'.
\end{cases}\] 
It was shown in \cite[Proof of Lemma 5.2]{HaglundHaimanLoehr2005} that 
\[ p(\si)+\maj(\si) = p\big(\phi(\si)\big)+\maj\big(\phi(\si)\big),\]
and
\[ \InvP(\si) = \InvP\big(\phi(\si)\big).\]
It implies that $G^\si_{\la^1,\dots,\la^r} =
G^{\phi(\si)}_{\la^1,\dots,\la^r}$, which finishes the proof.
\end{proof}

\begin{proof}[Proof of \cref{theo:basis}]
Observe that for any partitions
$\la^1,\dots,\la^r$ one has
\[ \ka(\la^1,\dots,\la^r)[\XX(t-1)]\in \QQ(q,t)\{s_\la\}_{\la \leq
    \la^{[r]}},\]
which is an immediate corollary from \ref{item1} and
\ref{item2}. Indeed, it is enough to use definition of Macdonald
cumulants (\eqref{eq:DefCumulants} and \eqref{eq:MacdoCumu}), and a well-known property of Schur functions: $s_\mu s_\nu \in
\Z\{s_\la\}_{\mu\cup\nu\leq \la \leq \mu\oplus\nu}$ (alternatively, it also follows
from \eqref{eq:cumulantPlethT}).

In particular, $\bar{H}_\mu[\XX(t-1);q,t] \in \QQ(q,t)\{s_\la\}_{\la \leq \mu}$
and it is enough to show that 
\[ [s_\mu]\bar{H}_\mu[\XX(t-1);q,t] \neq 0.\]
We claim that 
\begin{equation} 
\label{eq:takietam}
[s_\mu]\bar{H}_\mu[\XX(t-1);q,t] = (-1)^{|\mu|}P_{\mu^t_2, \mu^t_3,\dots,\mu^t_{\mu_1}}(q) +
  O(t),
\end{equation}
where $P_{\mu^t_2, \mu^t_3,\dots,\mu^t_{\mu_1}}$ is given by \eqref{eq:PolynomialP}.
Indeed, using \eqref{eq:cumulantPlethT} we end up with the following
expansion
\[ [s_\mu]\bar{H}_\mu[\XX(t-1);q,t] = (-1)^{|\mu|}\I_{G^\si_{1^{\mu^t_2}, 1^{\mu^t_3},\dots, 1^{\mu^t_{\mu_1}}}}(q) +
  O(t),\]
where $\si$ is the unique super filling compatible with $\mu$ with
entries
$\{\bar{1}^{\mu_1},\dots,\bar{\ell(\mu)}^{\mu_{\ell(\mu)}}\}$. This
filling is given by the explicit formula:
\[ \si(i,j) := \bar{j}\]
for any box $(i,j) \in \mu$. Thus, each pair of boxes lying in the
same row belongs to $\InvP(\si)$, since $\bar{i} \leq_- \bar{i} \leq_+
\overline{i-1}$. This implies that the number of edges $e_{i,j}\left(G^\si_{1^{\mu^t_2}, 1^{\mu^t_3},\dots,1^{\mu^t_{\mu_1}}}\right)$ linking vertices
$i\neq j $ is equal to the number
$\mu^t_{\max(i,j)}$, and \eqref{eq:takietam} follows from \cref{cor:specjalny}.
In particular $[s_\mu]\bar{H}_\mu[\XX(t-1);q,t] \neq 0$, which
finishes the proof.
\end{proof}

\begin{remark}
\label{rem:GParking}
Note that
\[ [s_{1^{|\mu|}}]\bar{H}_\mu = t^{n(\mu)}P_{\mu^t_2,
    \mu^t_3,\dots,\mu^t_{\mu_1}}(q),\]
where $n(\mu) := \sum_{i \geq 1}(i-1)\mu_i$. This follows from
\cref{theo:HookKostka} and from the above analysis of the polynomial $\I_{G^\si_{1^{\mu^t_2}, 1^{\mu^t_3},\dots, 1^{\mu^t_{\mu_1}}}}(q)$.
A polynomial $P_{a_1,\dots,a_r}(q)$ already appeared in the literature in the context
of Macdonald polynomials, and it would be interesting to find a
connection between these results and our work -- see \cref{sec:open} for more details.
\end{remark}

\begin{proposition}
For any partition $\mu$ the fully colored Macdonald polynomial
$\bar{H}_\mu(\xx;-1,t)$ is $t$--positive in the monomial basis, and
in the fundamental quasisymmetric basis.
\end{proposition}

\begin{proof}
This is straightforward from formulas \eqref{eq:MacdonaldCumuFormula}
and \eqref{eq:CumulantInQuasisymmetric} and from \cref{prop:TutteAnother}.
\end{proof}

\section{Open problems}
\label{sec:open}

We decided to conclude the paper by mentioning several possible directions for the future research that
arise naturally from \cref{theo:MacdonaldCumuFormula} which actually raises more questions than it answers.

\subsection{Schur positivity, $G$-parking functions and geometry of
  Hilbert schemes}
The first topic is related to the standard technique of proving
Schur--positivity of a given function $f$ by constructing a certain
$\Sym{n}$--module $V$ and interpreting $f$ as the \emph{Frobenius
  characteristic} of $V$. We are going to quickly review this
technique in the following. Let $V$ be a
$\Sym{n}$-module and we decompose it as a direct sum of its
irreducible submodules:
\[ V = \bigoplus_\la V_\la^{c_\la}.\]
Then, we define the Frobenius characteristic of $V$ as
\[ \Fr(V) := \sum_{\la}c_\la\ s_\la.\]
If, additionally, $V$ is a $k$-graded $\Sym{n}$--module, that is
\[ V = \bigoplus_{i_1,\dots,i_k \geq 0}V^{i_1,\dots,i_k},\]
and each summand in this decomposition is a $\Sym{n}$--module, then
\[ \Fr(V) := \sum_{i_1,\dots,i_k \geq 0} t_1^{i_1}\cdots t_k^{i_k}\
  \Fr\left(V^{i_1,\dots,i_k}\right).\] 
Equivalently, let $\mu \vdash n$ be a partition of $n$, and let
$V_\mu^{i_1,\dots,i_k}$ denote the subspace of $V^{i_1,\dots,i_k}$
consisting of fix-points of the action of the subgroup 
\[\Sym{\mu} :=
\Sym{\{1,\dots,\mu_1\}}\times\Sym{\{\mu_1+1,\dots,\mu_1+\mu_2\}}\times\cdots\times\Sym{\{\mu_1+\cdots+\mu_{l-1}+1,\dots,|\mu|\}}
< \Sym{n},\] 
where $\ell(\mu) = l$.
Then
\begin{equation} 
\label{eq:FrobeniusInMonomial}
\Fr(V) = \sum_{i_1,\dots,i_k \geq 0} t_1^{i_1}\cdots t_k^{i_k}\
  \sum_{\mu \vdash n}\dim\left(V^{i_1,\dots,i_k}_\mu\right)\ m_\mu.
\end{equation}

The celebrated result of Haiman
\cite{Haiman2001} proves that the Macdonald polynomial
$\tilde{H}_\mu(\xx;q,t)$ can be interpreted as the Frobenius
characteristic of a certain $\Sym{n}$--module $D_\mu$ (associated with
a partition $\mu \vdash n$),
which carries a natural structure of a bigraded module.
Thanks to our explicit, combinatorial formula
\eqref{eq:MacdonaldCumuFormula} it is natural to use \eqref{eq:FrobeniusInMonomial} and try to prove Schur--positivity
of Macdonald cumulant $\ka(\la_1,\dots,\la^r)$ by constructing a
bigraded $\Sym{n}$--module $D^{\la^1,\dots,\la^r}$ such that
\[ \Hilb_{q,t}\left( D^{\la^1,\dots,\la^r}_\mu \right) = \sum_{\si :
    \la^{[r]} \simeq
    \{1^{\mu_1},2^{\mu_2},\dots\}}\I_{G^\si_{\la^1,\dots,\la^r}}(q)\
  t^{\maj_{\la^{[r]}}(\si)},\]
where
\[ \Hilb_{q,t}(D) = \sum_{i,j}q^i\ t^j\ \dim(D^{i,j})\]
is the \emph{Hilbert series} of a bigraded vector space $D$ with
respect to its gradation. We mention here that Postnikov and Shapiro \cite{PostnikovShapiro2004}
introduced $G$-parking functions in order to construct certain graded
vector spaces, whose Hilbert series are given by $I_G(q)$. Is it
possible to merge ideas of Haiman, and Postnikov with Shapiro to
construct a module $D^{\la^1,\dots,\la^r}$ as in question?

In fact, Haiman's representation-theoretical interpretation of
Macdonald polynomials was a corollary of another result of him --
Haiman showed that a certain
geometric object, called \emph{isospectral Hilbert scheme} has
``nice'' geometric properties, that is it is \emph{normal, Cohen--Macaulay},
and \emph{Gorenstein} (see \cite{Haiman2002}, which explains all these
terms and much more in an available way for non-experts). What
kind of geometric properties (if any) of isospectral Hilbert schemes
or related geometric objects assure Schur--positivity of Macdonald
cumulants? The other way round -- does Schur--positivity of Macdonald
cumulants imply that some geometric object has nice properties? One
can ask a weaker question by using various specializations of
Macdonald cumulants carrying geometric interpretations.

\subsection{$G$-inversion polynomials and Macdonald polynomials}

We recall \cref{rem:GParking} which points that the coefficient
$[s_{1^{|\mu|}}]\bar{H}_\mu[\xx,q,t]$ is given by the polynomial $P_{\mu^t_2,
    \mu^t_3,\dots,\mu^t_{\mu_1}}(q)$, where
\[ P_{a_2,\dots,a_n}(q) := \sum_{T \text{ increasing tree on  }[n]
  }\ \prod_{2 \leq i \leq n} [\delta_T(i)]_q,\]
and $\delta_T(i) = \sum_j a_j$ (here $j$ ranges over descendants of
$i$ including $i$ itself ). When $a_2=\cdots=a_n = 1$ this is the
inversion polynomial but this also corresponds
to the generating function of parking functions with respect to the statistic
called \emph{area}. This function appeared in the context
of the Shuffle ex-conjecture
\cite{HaglundHaimanLoehrRemmelUlyanov2005} proved recently by Carlsson and Mellit \cite{CarlssonMellit2018}.

We define two sets $B_\mu(q,t) :=
\{q^{a'(\square)}t^{\ell'(\square)}| \square \in \mu\}, T_\mu(q,t) :=
B_\mu(q,t)\setminus\{1\}$. Given any symmetric function $f$, we define two operators $\Delta_f, \Delta'_f$ acting on the space of
symmetric functions$\Lambda$ by describing their action on the Macdonald basis
and extending this by linearity:
\[ \Delta_f \tilde{H}_\mu(\xx;q,t) := f (B_\mu(q,t))\cdot
  \tilde{H}_\mu(\xx;q,t), \quad \Delta_f \tilde{H}_\mu(\xx;q,t) := f (B_\mu(q,t))\cdot
  \tilde{H}_\mu(\xx;q,t).\]
The Shuffle ex-conjecture expresses the function
$\Delta'_{e_{n-1}}e_n$ in terms of parking functions with respect to
two statistics called \emph{area} and \emph{dinv} and the
polynomial $P_{1,\dots,1}(q)$ appears as $\langle
\Delta'_{e_{n-1}}e_n, h_{1^n}\rangle_{t=1}$, where $e_\mu,h_\mu$ are
elementary and complete symmetric functions, respectively.

The Shuffle ex-conjecture was generalized in two directions. The first generalization is given by the \emph{Rational
  Shuffle ex-conjecture} of Bergeron, Garsia, Leven, and Xin \cite{BergeronGarsiaLevenXin2015} proved
very recently by Mellit \cite{Mellit2016}, and the second one is given
by the \emph{Delta conjecture} of Haglund, Remmel and Wilson
\cite{HaglundRemmelWilson2018}. Both conjectures are related to the
combinatorics of \emph{Tesler matrices} and their generalizations, where polynomials $P_{a_2,\dots,a_n}(q)$ appear
naturally, see
\cite{ArmstrongGarsiaHaglundRhoadesSagan2012,Wilson2017}. Moreover,
polynomials $P_{a_2,\dots,a_n}(q)$ correspond to \emph{$(q, t)$--Ehrhart functions of
  certain flow polytopes} \cite{LiuMeszarosMorales2018}. 

We believe
that all these similarities are not coincidental and finding a missing
link between Macdonald cumulants and the aforementioned problems
would be of a great importance.

We leave all these questions wide open for future research.

\section*{Acknowledgments}
We would like to thank to the anonymous referees for their valuable comments.

\bibliographystyle{amsalpha}

\bibliography{biblio2015}

\end{document}